\documentclass[11pt]{amsart} 
\usepackage{amsmath,amssymb,amsthm} 
\usepackage[all]{xy} 

\newtheorem{theorem}{Theorem}[section] 
\newtheorem{lemma}[theorem]{Lemma} 
\newtheorem{proposition}[theorem]{Proposition} 
\newtheorem{corollary}[theorem]{Corollary} 
\newtheorem{maintheorem}{Theorem}

\theoremstyle{definition} 
\newtheorem{definition}[theorem]{Definition} 
\newtheorem{remark}[theorem]{Remark} 
\newtheorem{example}[theorem]{Example} 
\newtheorem{setting}[theorem]{Setting}

\numberwithin{equation}{section}

\renewcommand{\phi}{\varphi} 

\newcommand{\ep}{\varepsilon}

\newcommand{\Eff}{\operatorname{Eff}}

\newcommand{\Homeo}{\operatorname{Homeo}} 
\newcommand{\id}{\operatorname{id}} 
 
\newcommand{\Iso}{\operatorname{Iso}} 
\newcommand{\Ker}{\operatorname{Ker}}

\newcommand{\St}{\operatorname{St}} 
\newcommand{\supp}{\operatorname{supp}} 

\newcommand{\N}{\mathbb{N}} 
\newcommand{\Z}{\mathbb{Z}}

\newcommand{\cCO}{\mathcal{CO}} 
\newcommand{\cG}{\mathcal{G}} 
\newcommand{\cH}{\mathcal{H}} 
\newcommand{\cK}{\mathcal{K}} 
 
\newcommand{\cS}{\mathcal{S}} 

\newcommand{\sD}{\mathsf{D}} 
\newcommand{\sF}{\mathsf{F}}

\title[Stabilizers and maximal subgroups of topological full groups] 
{Stabilizers and maximal subgroups \\ 
of topological full groups} 
\author{Hiroki Matui} 
\thanks{The author was supported by JSPS KAKENHI Grant Number JP23K22397.} 
\address{Graduate School of Mathematical Sciences, 
The University of Tokyo, Japan} 
\email{hiroki@ms.u-tokyo.ac.jp} 
\date{}


\begin{document}

\begin{abstract} 
We study closed-set stabilizers in topological full groups of minimal 
effective ample groupoids with Cantor unit space. 
For almost finite or purely infinite groupoids, we extend the maximal 
stabilizer theorem of Grigorchuk and Vorobets to every subgroup 
containing the commutator subgroup. 
For a nonempty closed nowhere dense set, maximality of its stabilizer 
is equivalent to minimality of its stabilizer action, minimality of 
the regular reduction, and primeness of a canonical wide open subgroupoid. 
Without assuming almost finiteness or pure infiniteness, we prove 
spatial rigidity for stabilizers acting minimally on infinite closed 
nowhere dense sets. 
Their isomorphism types determine the ambient Kakutani equivalence 
class and, when Hausdorff, the effective reduction. 
We also realize every AF groupoid with compact unit space as 
the full restriction, with its relative topology, of any minimal 
effective ample groupoid with Cantor unit space to a closed 
nowhere dense subset. 
Combining rigidity and realization yields continuum many pairwise 
nonisomorphic stabilizers in every subgroup containing the commutator 
subgroup; these stabilizers are maximal when the ambient groupoid 
is almost finite or purely infinite. 
The realization also gives explicit AF subquotients of the ambient 
reduced groupoid $C^*$-algebra, including the CAR algebra, 
with compatibility on the canonical diagonal subalgebras. 
\end{abstract} 

\maketitle

\section{Introduction} 

Closed-set stabilizers provide a natural source of maximal subgroups 
of topological full groups. 
Their maximality is governed by the action on the closed set, 
while their abstract group structure can retain information 
both about the ambient dynamics and about the restricted action. 
In this paper, we study these stabilizers through maximality criteria, 
spatial rigidity and realization of prescribed reductions. 

For an effective ample groupoid $\cG$ with Cantor unit space $X$, 
we write $\sF(\cG)$ for its topological full group 
and $\sD(\cG)$ for the commutator subgroup of $\sF(\cG)$. 
All groupoids in this paper are second countable and Hausdorff. 
For a subgroup $\Gamma\subset\sF(\cG)$ and a closed subset $Y\subset X$, 
let 
\[ 
\St(\Gamma,Y):=\{\alpha\in\Gamma\mid\alpha(Y)=Y\} 
\] 
be the setwise stabilizer of $Y$, and put 
$\St(\cG,Y):=\St(\sF(\cG),Y)$. 
Our maximality results concern minimal groupoids 
that are either almost finite or purely infinite. 
These classes include groupoids of Cantor minimal $\Z$ systems 
and groupoids of irreducible shifts of finite type 
with non-permutation adjacency matrices, respectively. 

The starting point is the work of Grigorchuk and Vorobets 
on maximal subgroups of fully amplified groups \cite{GV26ETDS}. 
A group of homeomorphisms is fully amplified if it contains every 
homeomorphism locally agreeing with its elements. 
These groups are also called ample in \cite[Definition 2.5]{GV26ETDS}; 
we use fully amplified for groups and reserve ample for groupoids. 
They prove that, for a minimal fully amplified group with Property E 
acting on a Cantor set, the stabilizer of an infinite closed 
nowhere dense subset is maximal if and only if its action 
on that subset is minimal \cite[Theorem 1.3]{GV26ETDS}. 
Property E means that, for two intersecting clopen sets, 
the local subgroup on their union is generated 
by the local subgroups on the two sets. 
For a finite nonempty subset, maximality is equivalent 
to containment in a single orbit, under the weaker assumption 
that the fully amplified group has no finite orbits 
\cite[Theorem 1.2]{GV26ETDS}. 
For the groupoids considered here, $\sF(\cG)$ is fully amplified 
and has Property E, so these results already apply to $\sF(\cG)$ itself. 
We extend the criterion to every group 
\[ 
\sD(\cG)\subset\Gamma\subset\sF(\cG), 
\] 
including the commutator subgroup and the kernel of the index map, 
and express it in terms of reductions and prime subgroupoids. 
The groups properly contained in $\sF(\cG)$ in this range 
are not fully amplified. 

To describe the dynamics on an arbitrary closed set $Y$, 
we use the regular reduction $\cG|Y$. 
Its arrows are those with source and range in $Y$ 
that lie in a compact open bisection $U$ satisfying $YU=UY$ 
(Definition~\ref{def:etalesubsets}). 
This local preservation condition makes $\cG|Y$ an ample groupoid 
with the relative topology, although it need not be effective. 
It can be strictly smaller than the full restriction 
$r^{-1}(Y)\cap s^{-1}(Y)$. 
The two coincide precisely when $Y$ is $\cG$-\'etale, 
that is, when the full restriction is \'etale with the relative topology 
(see~\eqref{eq:etalesubsetregularity}). 
The union of all such $Y$-preserving bisections in $\cG$ 
also defines a canonical wide open subgroupoid $\cS_Y$, 
called the $Y$-preserving subgroupoid 
(Definition~\ref{def:preservingsubgroupoid}). 
It satisfies 
\[ 
\sF(\cS_Y)=\St(\cG,Y),\qquad \cS_Y|Y=\cG|Y, 
\] 
and, as sets, 
\[ 
\cS_Y=\cG|Y\sqcup\cG|(X\setminus Y). 
\] 
A proper wide open subgroupoid is called prime 
if there is no proper intermediate wide open subgroupoid 
\cite[Definition 5.1]{Ma26maximal1}. 

\begin{maintheorem}\label{intro:stabilizers} 
Let $\cG$ be a minimal effective ample groupoid 
with Cantor unit space $X$, 
and suppose that $\cG$ is either almost finite or purely infinite. 
Let $\Gamma$ be a group with $\sD(\cG)\subset\Gamma\subset\sF(\cG)$, 
and let $Y\subset X$ be a nonempty closed nowhere dense subset. 
Then the following conditions are equivalent. 
\begin{enumerate} 
\item $\St(\Gamma,Y)$ is a maximal subgroup of $\Gamma$. 
\item $\St(\Gamma,Y)$ acts minimally on $Y$. 
\item The regular reduction $\cG|Y$ is minimal. 
\item The inclusion $\cS_Y\subset\cG$ is prime. 
\end{enumerate} 
\end{maintheorem} 

This follows from Theorem~\ref{stabilizeractionmaximal}, 
Proposition~\ref{stabilizerminimality} and 
Proposition~\ref{prop:preservingsubgroupoid}. 
In particular, maximality holds simultaneously for all the groups 
$\Gamma$ in the statement, and $Y$ is then either finite or a Cantor set. 
The extension of Property E to these groups 
is proved in Proposition~\ref{propertyE}. 
The key observation is that every abelianization class 
of $\sF(\cG)$ can be represented by an element supported 
in any prescribed nonempty clopen set 
(Lemma~\ref{localabelianization}). 
This allows local factorizations in $\sF(\cG)$ 
to be adjusted to factorizations in $\Gamma$. 

The equivalence between minimality of $\cG|Y$ 
and primeness of $\cS_Y\subset\cG$ requires only minimality 
of the ambient groupoid. 
It relates closed-set stabilizers to the subgroupoid constructions 
in the companion paper \cite{Ma26maximal1}, 
which studies maximal subgroups arising from factor maps 
and wide open subgroupoids. 
In those constructions, conditions on the induced maps 
on abelianizations enter the maximality criteria. 
For a minimal almost finite or purely infinite $\cG$, 
the stabilizer $\St(\cG,Y)$ maps onto the abelianization 
of $\sF(\cG)$ for every proper closed set $Y$ 
(Remark~\ref{rem:stabilizerabelianization}). 
Thus the surjectivity condition is automatic here, 
and primeness alone characterizes maximality 
under the hypotheses of Theorem~\ref{intro:stabilizers}. 

For isomorphism questions, we also consider the effective reduction 
$\cG{\downarrow}Y:=\Eff(\cG|Y)$, 
obtained by quotienting out the interior isotropy. 
When this quotient is Hausdorff, restriction to $Y$ gives 
a surjective homomorphism 
\[ 
\St(\Gamma,Y)\longrightarrow\sF(\cG{\downarrow}Y) 
\] 
under the hypotheses of Theorem~\ref{intro:stabilizers} 
(Proposition~\ref{structureofSt}). 
In this case, minimality of $\cG{\downarrow}Y$ 
is another equivalent condition in that theorem. 
The Hausdorff assumption is needed for our use of the effective 
reduction, but not for the four conditions stated above. 
Under the corresponding Hausdorff assumptions, 
successive effective reductions agree with direct reduction 
when the intermediate closed set is $\cG$-\'etale 
(Proposition~\ref{prop:iteratedreduction}). 
The distinction between the regular reduction and the full 
restriction is already visible in stabilizers related to 
\cite[Example 6.9]{BBQS25TAMS}; 
an AF example is given in Example~\ref{ex:AFnonetalereduction}. 

Our next result shows that an abstract isomorphism of stabilizers 
is implemented on the ambient unit spaces and determines 
groupoid invariants both on and away from the closed sets. 
Neither spatial realization nor the groupoid invariants require 
an almost finite or purely infinite hypothesis. 

\begin{maintheorem}\label{intro:rigidity} 
For $i=1,2$, let $\cG_i$ be a minimal effective ample groupoid 
with Cantor unit space $X_i$, 
let $\Gamma_i$ be a group with 
$\sD(\cG_i)\subset\Gamma_i\subset\sF(\cG_i)$, 
and let $Y_i\subset X_i$ be an infinite closed nowhere dense subset 
on which $\St(\cG_i,Y_i)$ acts minimally. 
Let 
\[ 
\Phi:\St(\Gamma_1,Y_1)\longrightarrow\St(\Gamma_2,Y_2) 
\] 
be a group isomorphism. 
Then $\Phi$ is implemented by a homeomorphism $\phi:X_1\to X_2$ 
with $\phi(Y_1)=Y_2$, in the sense that 
\[ 
\Phi(\alpha)=\phi\circ\alpha\circ\phi^{-1} 
\qquad\text{for every }\alpha\in\St(\Gamma_1,Y_1). 
\] 
Moreover, $\cG_1$ and $\cG_2$ are Kakutani equivalent. 
If, in addition, each effective reduction 
$\cG_i{\downarrow}Y_i$ is Hausdorff, 
then $\cG_1{\downarrow}Y_1$ and $\cG_2{\downarrow}Y_2$ are isomorphic. 
\end{maintheorem} 

Theorem~\ref{intro:rigidity} combines Theorems~\ref{spatial} 
and \ref{thm:stabilizerinvariants}. 
The proof verifies the local conditions for spatial reconstruction 
from \cite{Ma15crelle} using elements of the ambient commutator subgroup. 
The action outside $Y$ recovers the Kakutani equivalence class 
of the ambient groupoid, while restriction to $Y$ 
recovers the effective reduction when it is Hausdorff. 
For the latter assertion, local three-cycles identify the germs 
of the restricted stabilizer action with the effective reduction. 
When $\cG$ is minimal and AF, these invariants give a complete 
isomorphism criterion for a natural class of closed sets. 
Let $Y_1,Y_2$ be $\cG$-\'etale Cantor sets that are thin, 
that is, null for every $\cG$-invariant probability measure, 
and suppose that $\cG|Y_1$ and $\cG|Y_2$ are minimal. 
Then their stabilizers in $\sF(\cG)$, and likewise in $\sD(\cG)$, 
are isomorphic if and only if there is an order isomorphism 
$H_0(\cG|Y_1)\to H_0(\cG|Y_2)$ 
sending $[1_{Y_1}]$ to $[1_{Y_2}]$ 
(Corollary~\ref{cor:thinAFrealization}). 

To construct stabilizers with prescribed invariants, 
we next ask which groupoids can occur as reductions 
to closed nowhere dense subsets. 
We prove the following realization theorem 
without any almost finite or purely infinite assumption 
on the ambient groupoid. 

\begin{maintheorem}\label{intro:realization} 
Let $\cG$ be a minimal effective ample groupoid 
with Cantor unit space $X$, 
and let $\cH$ be any AF groupoid with compact unit space. 
Then there exists a closed nowhere dense $\cG$-\'etale subset 
$Y\subset X$ such that 
\[ 
\cH\cong\cG|Y=r^{-1}(Y)\cap s^{-1}(Y), 
\] 
where the full restriction has the relative topology. 
\end{maintheorem} 

This is Corollary~\ref{cor:AFrealizationminimal}. 
We first realize the binary tail equivalence groupoid 
$\mathcal R_{2^\infty}$ inside $\cG$, 
and then use AF realization inside this groupoid 
\cite[Lemma 3.1]{Ma08JMSJ}. 
The first step is related to the nested clopen construction 
of Grigorchuk and Vorobets 
\cite[Lemma 7.1 and Proposition 7.2]{GV26ETDS}. 
In the first step, we also control every ambient arrow 
with both endpoints in the constructed Cantor set, 
identifying the full restriction with $\mathcal R_{2^\infty}$. 
After the second step, the full restriction to the resulting set $Y$ 
is isomorphic to $\cH$. 
The groupoid $\cH$ need not be minimal. 
For a Cantor minimal $\Z$ system, the converse holds as well: 
every regular reduction to a closed nowhere dense subset is AF 
(Proposition~\ref{prop:ZAFrealization}). 

Combining realization with rigidity gives $2^{\aleph_0}$ 
pairwise nonisomorphic stabilizers. 
More precisely, for every minimal effective ample groupoid $\cG$ 
with Cantor unit space, 
there is a family of nowhere dense $\cG$-\'etale Cantor sets 
$(Y_t)_{t\in\mathcal T}$, with $\#\mathcal T=2^{\aleph_0}$, 
such that, for every fixed 
$\sD(\cG)\subset\Gamma\subset\sF(\cG)$, 
the groups $\St(\Gamma,Y_t)$ are pairwise nonisomorphic. 
If $\cG$ is also almost finite or purely infinite, 
all these stabilizers are maximal in $\Gamma$ 
(Proposition~\ref{prop:uncountablestabilizers}). 
We obtain this family by prescribing UHF groupoids 
with different supernatural types. 
Grigorchuk and Vorobets construct uncountably many distinct 
closed sets with minimal stabilizer actions 
\cite[Proposition 7.2]{GV26ETDS}; 
the reduction invariants here distinguish 
the resulting stabilizers up to abstract group isomorphism. 

Theorem~\ref{intro:realization} also has a $C^*$-algebraic consequence. 
For every prescribed AF groupoid $\cH$ with compact unit space, 
the realization $\cH\cong\cG|Y$ gives an inclusion and a quotient 
\[ 
C_r^*(\cS_Y)\subset C_r^*(\cG),\qquad 
C_r^*(\cS_Y)\longrightarrow C_r^*(\cH), 
\] 
where the second map is surjective and carries 
the canonical diagonal $C(X)$ onto $C(\cH^{(0)})$ 
(Corollary~\ref{cor:AFsubquotients}). 
In particular, the CAR algebra $M_{2^\infty}$ occurs 
as a $C^*$-subquotient of $C_r^*(\cG)$. 
Its existence already follows abstractly from Glimm's theorem 
\cite{Gl61Annals}: Renault's simplicity theorem 
\cite[Corollary 4.6]{Re91JOT} makes $C_r^*(\cG)$ 
simple, and this separable unital infinite-dimensional algebra 
is not of type I. 
The realization above specifies an intermediate groupoid algebra 
and a restriction quotient for every such $\cH$, 
with compatibility on the canonical diagonals. 
It requires no amenability assumption on $\cG$. 

Belk, Bleak, Quick and Skipper \cite{BBQS25TAMS} use type systems 
to construct maximal subgroups of Thompson's group $V$, 
including an uncountable family of pairwise nonisomorphic examples. 
Their framework also includes closed-set stabilizers. 
The companion paper \cite{Ma26maximal3} studies wide open subgroupoids 
of irreducible SFT groupoids using subgroupoid graphs 
and models by countable graph groupoids. 
It proves a trichotomy for prime subgroupoids and develops 
maximality criteria for their topological full groups and normalizers, 
including the case of minimal subgroupoids. 
The $Y$-preserving subgroupoids studied here account for 
the case with a nowhere dense invariant set. 
That paper also extends the type-system constructions 
of \cite{BBQS25TAMS} to general irreducible SFT groupoids 
and constructs continuum many pairwise nonisomorphic 
maximal subgroups acting minimally on the whole unit space. 

Section~\ref{sec:preliminaries} recalls preliminaries on ample groupoids 
and their topological full groups. 
Section~\ref{sec:stabilizer} introduces the reductions, 
establishes Property E and proves the maximality criterion 
together with its relation to prime $Y$-preserving subgroupoids. 
Section~\ref{sec:stabilizerinvariants} proves spatial rigidity 
and the isomorphism invariants, and then develops realization 
and its consequences for maximal stabilizers and AF subquotients.

\section{Preliminaries}\label{sec:preliminaries} 

\subsection{Ample groupoids}\label{subsec:amplegroupoids} 

We fix notation and recall the basic terminology for ample groupoids. 

The cardinality of a set $A$ is written $\#A$ and 
the characteristic function of $A$ is written $1_A$. 
We say that a subset of a topological space $X$ is clopen 
if it is both closed and open. 
We let $\cCO(X)$ denote 
the set of all nonempty compact open subsets of $X$. 
A topological space is said to be totally disconnected 
if its connected components are singletons. 
By a Cantor set, 
we mean a nonempty compact, metrizable, totally disconnected space 
with no isolated points. 
It is known that any two such spaces are homeomorphic. 

In this paper, by an \'etale groupoid 
we mean a second countable locally compact Hausdorff groupoid 
such that the range map is a local homeomorphism. 
(We emphasize that our \'etale groupoids are always assumed to be Hausdorff, 
while non-Hausdorff groupoids are also studied actively.) 
We refer the reader to \cite{Re_text,Re08Irish} 
for background material on \'etale groupoids. 
For an \'etale groupoid $\cG$, 
we let $\cG^{(0)}$ denote the unit space and 
let $s$ and $r$ denote the source and range maps, 
i.e.\ $s(g)=g^{-1}g$, $r(g)=gg^{-1}$. 
A subset $U\subset\cG$ is called a bisection 
if $r|U$ and $s|U$ are injective. 
The \'etale groupoid $\cG$ has a basis for its topology 
consisting of open bisections. 
For $x\in\cG^{(0)}$, 
the set $r(s^{-1}(x))$ is called the orbit 
(or $\cG$-orbit) of $x$. 
When every orbit is dense in $\cG^{(0)}$, $\cG$ is said to be minimal. 
An \'etale groupoid $\cG$ is called ample 
if its unit space $\cG^{(0)}$ is totally disconnected. 
An \'etale groupoid is ample if and only if 
it has a basis for its topology consisting of compact open bisections. 

For $x\in\cG^{(0)}$, 
we write $\cG_x=r^{-1}(x)\cap s^{-1}(x)$ and 
call it the isotropy group of $x$. 
The isotropy subgroupoid of $\cG$ is 
$\Iso(\cG):=\{g\in\cG\mid r(g)=s(g)\}=\bigsqcup_{x\in\cG^{(0)}}\cG_x$. 
The interior $\Iso(\cG)^\circ$ is an open subgroupoid of $\cG$. 
We say that $\cG$ is principal if $\Iso(\cG)=\cG^{(0)}$. 
We say that $\cG$ is effective 
if the interior of $\Iso(\cG)$ is $\cG^{(0)}$. 
Note that the term essentially principal is also used 
in the literature. 
In this paper, unless specified otherwise, 
an ample groupoid is always assumed to be effective. 
A subgroupoid $\cH\subset\cG$ is said to be wide 
if $\cH^{(0)}=\cG^{(0)}$. 
A wide open subgroupoid of an ample groupoid is again ample, 
and is effective whenever the ambient groupoid is effective. 

For an ample groupoid $\cG$ and a clopen subset $Y\subset\cG^{(0)}$, 
we write 
\[ 
\cG|Y:=r^{-1}(Y)\cap s^{-1}(Y), 
\] 
which is again an ample groupoid, with unit space $Y$. 
Since $\cG$ is effective, so is $\cG|Y$. 
We say that $Y$ is full if $Y$ meets every $\cG$-orbit. 
Two ample groupoids $\cG_1$ and $\cG_2$ are said to be 
Kakutani equivalent 
if there exist full clopen subsets $Y_i\subset\cG_i^{(0)}$ 
such that $\cG_1|Y_1$ and $\cG_2|Y_2$ are isomorphic. 
In particular, $\cG$ and $\cG|Y$ are Kakutani equivalent 
for every full clopen subset $Y\subset\cG^{(0)}$.

\subsection{Topological full groups, homology and the index map} 
\label{subsec:homology} 

For a bisection $U\in\cCO(\cG)$, we write 
$\theta_U:=r\circ(s|U)^{-1}:s(U)\to r(U)$ 
for the associated homeomorphism. 
The bisection $U$ is said to be full 
when $r(U)=\cG^{(0)}=s(U)$; in this case 
$\theta_U\in\Homeo(\cG^{(0)})$. 

\begin{definition}[{\cite[Definition 2.3]{Ma12PLMS}}] 
\label{def:topologicalfullgroup} 
Let $\cG$ be an effective ample groupoid 
whose unit space $\cG^{(0)}$ is compact. 
The set of all $\alpha\in\Homeo(\cG^{(0)})$, 
for which there exists a full bisection $U\in\cCO(\cG)$ 
such that $\alpha=\theta_U$, 
is called the topological full group of $\cG$ 
and is denoted by $\sF(\cG)$. 
\end{definition} 

For $\alpha\in\sF(\cG)$ 
the compact open bisection $U$ as above uniquely exists, 
because $\cG$ is effective. 
Obviously $\sF(\cG)$ is a subgroup of $\Homeo(\cG^{(0)})$. 
Since $\cG$ is second countable, it has countably many compact open subsets, 
and so $\sF(\cG)$ is at most countable. 
A homeomorphism $\alpha:\cG^{(0)}\to\cG^{(0)}$ belongs to $\sF(\cG)$ 
if and only if for any $x\in\cG^{(0)}$ 
there exists a bisection $V\in\cCO(\cG)$ 
such that $x$ is in $s(V)$ and 
$\alpha$ equals $r\circ(s|V)^{-1}$ on a neighborhood of $x$. 
Thus, $\alpha$ is in $\sF(\cG)$ if and only if 
the `graph' of $\alpha$ is a clopen subset of $\cG$. 
We let $\sD(\cG)$ denote the derived subgroup (commutator subgroup) 
of $\sF(\cG)$.

For $\alpha\in\Homeo(\cG^{(0)})$ we write 
\[ 
\supp(\alpha):=\overline{\{x\in\cG^{(0)}\mid\alpha(x)\neq x\}}. 
\] 
When $\alpha=\theta_U$ belongs to $\sF(\cG)$, 
one has $\supp(\alpha)=s(U\setminus\cG^{(0)})$, 
and so $\supp(\alpha)$ is a clopen subset of $\cG^{(0)}$ 
by \cite[Lemma 2.2]{Ma15crelle}. 

\begin{definition}\label{def:localsubgroup} 
Let $\Gamma$ be a subgroup of $\sF(\cG)$ and 
let $V\subset\cG^{(0)}$ be a clopen subset. 
We call 
\[ 
\Gamma_V:=\{\alpha\in\Gamma\mid\supp(\alpha)\subset V\} 
\] 
the local subgroup of $\Gamma$ on $V$. 
\end{definition} 

Extending homeomorphisms by the identity outside $V$, we identify 
$\sF(\cG|V)$ with $\sF(\cG)_V$. 
Under this identification, 
$\sD(\cG|V)\subset\sD(\cG)_V$. 
We note that $\Gamma_{\beta(V)}=\beta\Gamma_V\beta^{-1}$ 
for any $\beta\in\sF(\cG)$ normalizing $\Gamma$ and 
any clopen set $V$, because 
$\supp(\beta\alpha\beta^{-1})=\beta(\supp(\alpha))$. 

We write $H_n(\cG)$ for the groupoid homology with integer 
coefficients and refer to \cite[Section 3]{Ma12PLMS} for its 
definition and basic properties. 
A compact open subset $A\subset\cG^{(0)}$ determines a class 
$[1_A]\in H_0(\cG)$, with 
\[ 
[1_{A\sqcup B}]=[1_A]+[1_B],\qquad 
[1_{r(U)}]=[1_{s(U)}] 
\] 
for disjoint compact open sets $A,B$ and compact open bisections $U$. 
In fact, $H_0(\cG)$ is generated by the classes of compact open sets, 
subject to these relations. 
The groups $H_n(\cG)$ are invariants of groupoid isomorphism 
and Kakutani equivalence. 
For an AF groupoid, $H_0$ is its dimension group; 
see \cite[Theorem 4.10]{Ma12PLMS}. 

For $\alpha\in\sF(\cG)$, let $U$ be its unique compact open 
full bisection. 
The assignment 
\[ 
I:\sF(\cG)\longrightarrow H_1(\cG),\qquad I(\alpha):=[1_U] 
\] 
is a homomorphism, called the index map 
\cite[Section 7]{Ma12PLMS}. 
In particular, $\sD(\cG)\subset\Ker I$. 
For an abstract group $\Lambda$, we use 
$H_1(\Lambda)=\Lambda/[\Lambda,\Lambda]$ for its abelianization.

\subsection{Almost finite and purely infinite groupoids} 

We recall the classes of groupoids and local properties 
used in the stabilizer arguments. 

For an ample groupoid $\cG$ with compact unit space, 
a compact open principal subgroupoid $\cK\subset\cG$ 
with $\cK^{(0)}=\cG^{(0)}$ is called an elementary subgroupoid. 
We write $\cK x:=\{g\in\cK\mid s(g)=x\}$. 
We say that $\cG$ is AF if it is the union of an increasing 
sequence of elementary subgroupoids. 
In particular, every AF groupoid is principal. 

\begin{definition}[{\cite[Definition 6.2]{Ma12PLMS}}]\label{af/def} 
Let $\cG$ be an ample groupoid with $\cG^{(0)}$ compact. 
We say that $\cG$ is almost finite 
if for any compact subset $C\subset\cG$ and $\ep>0$ 
there exists an elementary subgroupoid $\cK\subset\cG$ such that 
\[ 
\frac{\#(C\cK x\setminus\cK x)}{\#(\cK x)}<\ep 
\] 
for all $x\in\cG^{(0)}$. 
\end{definition} 

Thus the finite fibers $\cK x$ are approximately invariant 
under multiplication by $C$. 
Every AF groupoid is almost finite. 

\begin{definition}[{\cite[Definition 4.9]{Ma15crelle}}]\label{pi/def} 
Let $\cG$ be an ample groupoid with $\cG^{(0)}$ compact. 
We say that $\cG$ is purely infinite 
if for every clopen set $A\subset\cG^{(0)}$, 
there exist compact open bisections $U,V\subset\cG$ 
such that $s(U)=s(V)=A$, $r(U)\sqcup r(V)\subset A$. 
\end{definition} 

Almost finiteness is invariant under Kakutani equivalence 
for ample groupoids with compact unit spaces 
by \cite[Proposition 4.6]{ABBL23ETDS}. 
If $\cG$ is minimal and either almost finite or purely infinite, 
then so is $\cG|W$ for every $W\in\cCO(\cG^{(0)})$: 
in the almost finite case this follows since $W$ is full, 
and in the purely infinite case directly from the definition. 

We recall two properties used in Section~\ref{subsec:propertyE}. 

\begin{theorem}\label{niceproperties} 
Let $\cG$ be a minimal ample groupoid with Cantor unit space 
which is either almost finite or purely infinite. 
\begin{enumerate} 
\item Any non-trivial subgroup of $\sF(\cG)$ normalized by 
$\sD(\cG)$ contains $\sD(\cG)$. 
In particular, $\sD(\cG)$ is simple. 
\item The groupoid $\cG$ has cancellation: for nonempty clopen 
sets $A,B\subset\cG^{(0)}$ with $[1_A]=[1_B]$ in $H_0(\cG)$, 
there exists a compact open bisection $U$ with $r(U)=A$ and $s(U)=B$. 
\end{enumerate} 
\end{theorem} 

\begin{proof} 
The first assertion follows from 
\cite[Theorems 4.7 and 4.16]{Ma15crelle}. 
The second follows from \cite[Theorem 6.12]{Ma12PLMS} and 
\cite[Theorem 4.2 (3)]{Ma16Adv}. 
\end{proof} 

We next recall transpositions and three-cycles. 
For a compact open bisection $U$ with $r(U)\cap s(U)=\emptyset$, 
define $\tau_U:=\theta_W\in\sF(\cG)$, where 
\[ 
W:=U\sqcup U^{-1}\sqcup 
\left(\cG^{(0)}\setminus(r(U)\cup s(U))\right). 
\] 
Thus $\tau_U$ is a transposition, with $\tau_U^2=1$ and 
$\supp(\tau_U)=r(U)\sqcup s(U)$. 
We also record two local constructions that do not require 
almost finiteness or pure infiniteness. 

\begin{lemma}\label{threecycle} 
Let $\cH$ be an ample groupoid with $\cH^{(0)}$ a Cantor set 
and let $U,V\in\cCO(\cH)$ be bisections with $s(U)=s(V)$ 
such that $s(U)$, $r(U)$ and $r(V)$ are mutually disjoint. 
\begin{enumerate} 
\item The element $\tau_V\tau_U$ equals the commutator $[\tau_U,\tau_V]$ 
and has order three, with 
\[ 
\supp(\tau_V\tau_U)=s(U)\cup r(U)\cup r(V). 
\] 
In particular $\tau_V\tau_U\in\sD(\cH)$. 
\item The element $\tau_V\tau_U$ agrees with $\theta_U$ on $s(U)$. 
\item Suppose in addition that $\cH$ is minimal. 
For any $x\in\cH^{(0)}$ and any $A\in\cCO(\cH^{(0)})$, 
there exists $\sigma\in\sD(\cH)$ such that $\sigma(x)\in A$. 
Consequently, for any $B\in\cCO(\cH^{(0)})$, 
the local subgroup $\sD(\cH)_B$ acts minimally on $B$, 
and $\sD(\cH)$ acts minimally on $\cH^{(0)}$. 
\end{enumerate} 
\end{lemma} 

\begin{proof} 
(1) and (2) are clear. 

(3)\: 
We may assume $x\notin A$. 
By minimality there exists $g\in\cH$ with $s(g)=x$ and $r(g)\in A$. 
Choose a bisection $U\in\cCO(\cH)$ with $g\in U$, 
$r(U)\subset A$, $r(U)\cap s(U)=\emptyset$ and 
$s(U)\cup r(U)\neq\cH^{(0)}$; 
the last condition can be arranged by shrinking $s(U)$, 
because $\cH^{(0)}$ has no isolated points. 
Applying minimality again, choose an arrow with source $x$ 
and range in $\cH^{(0)}\setminus(s(U)\cup r(U))$, 
and a compact open bisection $V$ containing it with range in this complement. 
Restricting $U$ and $V$ to a common clopen neighborhood of $x$ 
in $s(U)\cap s(V)$, we may assume $s(U)=s(V)$. 
Then $\sigma:=\tau_V\tau_U$ belongs to $\sD(\cH)$ by (1) 
and satisfies $\sigma(x)=\theta_U(x)\in A$ by (2). 
The last assertion follows by applying this to $\cH|B$. 
\end{proof} 

\begin{lemma}\label{involution} 
Let $\cG$ be an ample groupoid with compact unit space, 
and let $V_1,V_2\in\cCO(\cG)$ be bisections 
satisfying $r(V_j)\cap s(V_j)=\emptyset$ for $j=1,2$. 
Suppose that $\supp(\tau_{V_1})$ and $\supp(\tau_{V_2})$ are disjoint 
and that there is a compact open bisection $Z$ 
with $s(Z)=s(V_1)$ and $r(Z)=s(V_2)$. 
Then $\tau_{V_1}\tau_{V_2}$ is an involution belonging to $\sD(\cG)$, 
and 
\[ 
\supp(\tau_{V_1}\tau_{V_2}) 
=\supp(\tau_{V_1})\sqcup\supp(\tau_{V_2}) 
=s(V_1)\sqcup r(V_1)\sqcup s(V_2)\sqcup r(V_2). 
\] 
\end{lemma} 

\begin{proof} 
The bisection $Z$ shows that $\tau_{V_1}$ and $\tau_{V_2}$ 
are conjugate involutions in $\sF(\cG)$, 
so their product belongs to $\sD(\cG)$. 
The remaining assertions are clear. 
\end{proof}

The following two lemmas are proved in \cite{Ma26maximal1}. 
We recall their statements for use below. 

\begin{lemma}[{\cite[Lemma 2.9]{Ma26maximal1}}] 
\label{localabelianization} 
Let $\cG$ be a minimal ample groupoid with Cantor unit space 
which is either almost finite or purely infinite. 
For every nonempty clopen set $W\subset\cG^{(0)}$, one has 
\[ 
\sF(\cG)=\sF(\cG)_W\cdot\sD(\cG). 
\] 
Equivalently, the inclusion $\sF(\cG|W)\to\sF(\cG)$ induces 
a surjection on abelianizations. 
\end{lemma} 

\begin{lemma}[{\cite[Lemma 2.10 (2)]{Ma26maximal1}}] 
\label{elementofDG} 
Let $\cG$ be a minimal ample groupoid with Cantor unit space 
which is either almost finite or purely infinite. 
If a bisection $U\in\cCO(\cG)$ satisfies 
$r(U)\neq\cG^{(0)}$ and $s(U)\neq\cG^{(0)}$, then 
there exists a full bisection $V$ such that 
$U\subset V$ and $\theta_V\in\sD(\cG)$. 
\end{lemma}

\section{Maximal subgroups arising from stabilizers}\label{sec:stabilizer} 

Throughout this section, 
we let $\cG$ be an (effective) ample groupoid 
with $\cG^{(0)}$ a Cantor set. 

When $Y$ is a subset of $\cG^{(0)}$, 
the stabilizer $\St(\cG,Y)$ is defined by 
\[ 
\St(\cG,Y):=\{\alpha\in\sF(\cG)\mid\alpha(Y)=Y\}. 
\] 
For any subgroup $\Gamma\subset\sF(\cG)$, we also write 
\[ 
\St(\Gamma,Y):=\Gamma\cap\St(\cG,Y) 
=\{\alpha\in\Gamma\mid\alpha(Y)=Y\}. 
\] 
Thus $\St(\sF(\cG),Y)=\St(\cG,Y)$. 

Grigorchuk and Vorobets proved that, 
for a fully amplified group acting minimally on a Cantor set 
and having Property E, 
the stabilizer of an infinite closed nowhere dense set is maximal 
if and only if its action on that set is minimal 
\cite[Theorem 1.3]{GV26ETDS}. 
We extend this criterion to all groups between $\sD(\cG)$ and $\sF(\cG)$ 
when $\cG$ is minimal and either almost finite or purely infinite. 
We first define regular and effective reductions to closed subsets. 
Section~\ref{subsec:propertyE} establishes Property E 
for the intermediate groups considered here, 
using cancellation and Lemma~\ref{localabelianization}. 
Section~\ref{subsec:maximality} proves the maximality criterion 
and relates it to primeness of the $Y$-preserving subgroupoid. 
Section~\ref{sec:stabilizerinvariants} then studies isomorphism 
invariants of these stabilizers and uses AF realizations to construct 
pairwise nonisomorphic maximal subgroups.

\subsection{Reduction to closed subsets}\label{subsec:closedreduction} 

When $Y$ is an open subset of $\cG^{(0)}$, 
the subset $r^{-1}(Y)\cap s^{-1}(Y)$ with the relative topology 
becomes an ample groupoid in a natural way, 
which is called the reduction of $\cG$ to $Y$ in the literature. 
For a closed subset $Y\subset\cG^{(0)}$, however, 
the subgroupoid $r^{-1}(Y)\cap s^{-1}(Y)$ may not be \'etale. 
We first retain the arrows 
which locally preserve $Y$ to obtain an ample groupoid, 
and then take its effective quotient. 
For the latter step, we will impose a Hausdorffness condition. 

\begin{definition}\label{def:etalesubsets} 
A closed subset $Y\subset\cG^{(0)}$ is called $\cG$-\'etale if 
the full restriction $r^{-1}(Y)\cap s^{-1}(Y)$ is \'etale with 
the relative topology; see \cite[Section 5.3]{Ma16Adv}. 
For an arbitrary subset $Y\subset\cG^{(0)}$, we say that 
$g\in r^{-1}(Y)\cap s^{-1}(Y)$ is regular 
if there exists a bisection $U\in\cCO(\cG)$ such that 
$g\in U$ and $YU=UY$. 
Define 
\[ 
\cG|Y:= 
\left\{g\in r^{-1}(Y)\cap s^{-1}(Y)\mid\text{$g$ is regular}\right\}. 
\] 
\end{definition} 

For a bisection $U$, the equality $YU=UY$ is equivalent to 
\[ 
r(h)\in Y\iff s(h)\in Y 
\qquad\text{for every }h\in U. 
\] 
In particular, this condition is inherited by every sub-bisection of $U$. 
If $Y$ is open, then any $g\in r^{-1}(Y)\cap s^{-1}(Y)$ is regular, 
and hence $\cG|Y=r^{-1}(Y)\cap s^{-1}(Y)$. 
In particular, $\cG|Y$ is an ample groupoid with the relative topology.

\begin{lemma} 
If $Y\subset\cG^{(0)}$ is closed, then $\cG|Y$ is an ample groupoid 
with the relative topology, 
which is not necessarily effective. 
\end{lemma} 

\begin{proof} 
It is easy to see that 
$\cG|Y$ is a second countable locally compact Hausdorff groupoid, 
because $\cG|Y$ is an open subset of $r^{-1}(Y)\cap s^{-1}(Y)$. 
The unit space of $\cG|Y$ is $Y$, which is totally disconnected. 
It remains to show that $\cG|Y$ is \'etale. 
For any $g\in\cG|Y$, 
there exists a bisection $U\in\cCO(\cG)$ 
such that $g\in U$ and $YU=UY$. 
For every $h\in U$, we have $r(h)\in Y$ if and only if $h\in\cG|Y$. 
The homeomorphism $r|U$ therefore restricts to a homeomorphism 
from $U\cap\cG|Y$ onto the open subset $r(U)\cap Y$ of $Y$. 
Hence $\cG|Y$ is \'etale. 
\end{proof} 

By \cite[Lemma 5.9]{Ma16Adv}, a closed subset $Y$ is $\cG$-\'etale 
if and only if every arrow with source and range in $Y$ is regular. 
Equivalently, 
\begin{equation}\label{eq:etalesubsetregularity} 
Y\text{ is }\cG\text{-\'etale} 
\iff \cG|Y=r^{-1}(Y)\cap s^{-1}(Y). 
\end{equation} 

\begin{remark}\label{rem:etalesubsets} 
For \'etale equivalence relations, \'etale subsets play an important 
role in the work of Giordano, Putnam and Skau 
\cite[Theorem 3.11]{GPS04ETDS} and of Giordano, Matui, Putnam and 
Skau \cite[Definition 4.1]{GMPS08ETDS}. 
Their notion of an $R$-\'etale subset requires the full restriction 
$R\cap(Y\times Y)$ to be \'etale with the relative topology. 
The definition for ample groupoids in \cite[Section 5.3]{Ma16Adv} 
extends this notion. 
\end{remark} 

\begin{example}\label{ex:fullrestriction} 
\begin{enumerate} 
\item It is easy to find $Y$ 
satisfying $\cG|Y\neq r^{-1}(Y)\cap s^{-1}(Y)$. 
Let $g\in\cG$ satisfy $r(g)\neq s(g)$. 
Choose a clopen neighborhood $A$ of $r(g)$ not containing $s(g)$, 
and a sequence of distinct points $x_n\in A\setminus\{r(g)\}$ 
converging to $r(g)$. 
Set $Y:=\{x_n\mid n\in\N\}\cup\{r(g),s(g)\}$. 
Then $r(g)$ is nonisolated in $Y$, whereas $s(g)$ is isolated. 
Hence $g\in r^{-1}(Y)\cap s^{-1}(Y)$ is not regular. 
\item In general, $\cG|Y$ may not be effective 
even if $\cG$ is effective. 
Indeed, suppose that 
there exists $x\in\cG^{(0)}$ such that $\cG_x\neq\{x\}$. 
Consider the effective ample groupoid $\cH:=\cG\times\cG$ and 
set $Y:=\{x\}\times\cG^{(0)}\subset\cH^{(0)}$. 
Then $\cH|Y=r^{-1}(Y)\cap s^{-1}(Y)=\cG_x\times\cG$ is 
an ample groupoid which is not effective. 
\end{enumerate} 
\end{example}

The full restriction can differ from the regular reduction even for 
nowhere dense Cantor sets on which the stabilizer acts minimally. 
Our introduction of regular reductions was motivated by 
\cite[Example 6.9]{BBQS25TAMS}; see \cite[Example 8.5]{Ma26maximal3} 
for its groupoid description. 
Example~\ref{ex:AFnonetalereduction} below gives an AF variant, 
in which passing to the regular reduction changes the orbits.

\begin{definition}\label{def:effectivequotient} 
Let $\cH$ be a Hausdorff ample groupoid, not necessarily effective. 
Following the terminology of \cite[p.~226]{Ne22GSM}, 
we call 
\[ 
\Eff(\cH):=\cH/\Iso(\cH)^\circ 
\] 
the effective quotient of $\cH$. 
Here the quotient is equipped with the quotient topology, and 
\[ 
g\sim h 
\iff s(g)=s(h)\ \text{and}\ gh^{-1}\in\Iso(\cH)^\circ. 
\] 
Equivalently, 
$\Eff(\cH)$ is the groupoid of germs of the local homeomorphisms 
of $\cH^{(0)}$ induced by open bisections of $\cH$. 
\end{definition} 

The quotient map $\pi\colon\cH\to\Eff(\cH)$ is open, 
it identifies the unit space of $\Eff(\cH)$ with $\cH^{(0)}$, 
and it maps every open bisection of $\cH$ 
homeomorphically onto an open bisection of $\Eff(\cH)$. 
Consequently $\Eff(\cH)$ is second countable, 
its range map is a local homeomorphism, 
it has a basis consisting of compact open bisections, 
and the interior of its isotropy subgroupoid is its unit space. 
Thus $\Eff(\cH)$ is an effective ample groupoid in our Hausdorff 
convention if and only if $\Iso(\cH)^\circ$ is closed in $\cH$. 
In terms of the associated pseudogroup, 
this is the condition that 
the interior of the fixed point set of each local homeomorphism 
is closed in its domain; see \cite[Proposition 3.1.22]{Ne22GSM}. 
Accordingly, for a closed subset $Y\subset\cG^{(0)}$, we impose the condition 
\begin{equation}\label{eq:hausdorffreduction} 
\Iso(\cG|Y)^\circ\ \text{is closed in }\cG|Y. 
\end{equation} 
This condition ensures that 
$\Eff(\cG|Y)$ is a Hausdorff effective ample groupoid, 
in accordance with our conventions. 

\begin{definition} 
Let $Y\subset\cG^{(0)}$. 
\begin{enumerate} 
\item When $Y$ is open, we call 
$\cG{\downarrow}Y=\cG|Y=r^{-1}(Y)\cap s^{-1}(Y)$ 
the reduction of $\cG$ to $Y$. 
\item When $Y$ is closed, we call 
$\cG|Y$ the regular reduction of $\cG$ to $Y$. 
\item When $Y$ is closed and satisfies \eqref{eq:hausdorffreduction}, 
we call $\cG{\downarrow}Y:=\Eff(\cG|Y)$ 
the effective reduction of $\cG$ to $Y$. 
\end{enumerate} 
\end{definition} 

For open $Y$, the reduction $\cG|Y$ is an open subgroupoid of the effective 
groupoid $\cG$, hence is effective; this justifies the equality in (1). 
For closed $Y$, passing to the effective quotient identifies arrows 
whose induced local homeomorphisms have the same germ on $Y$. 
Whenever an effective reduction to a closed subset is used below, 
we assume that the closed subset satisfies \eqref{eq:hausdorffreduction} 
with respect to its ambient groupoid. 
This convention imposes no additional assumption on statements 
which only concern the stabilizer action or the regular reduction. 

The Hausdorffness condition is automatic 
in two basic cases. 
First, if $\cG$ is principal, then $\cG|Y$ is principal 
for every closed $Y$, and hence $\cG{\downarrow}Y=\cG|Y$. 
This applies, in particular, 
to groupoids of minimal $\Z$ actions on Cantor sets. 
Second, the condition is automatic when $Y$ is finite. 
In this case $Y$ is discrete in the relative topology. 
Given $g\in r^{-1}(Y)\cap s^{-1}(Y)$, 
choose clopen sets $A\ni r(g)$ and $B\ni s(g)$ 
with $A\cap Y=\{r(g)\}$ and $B\cap Y=\{s(g)\}$, 
and a bisection $U\in\cCO(\cG)$ 
with $g\in U$, $r(U)\subset A$ and $s(U)\subset B$. 
For $h\in U$ we have $r(h)\in Y$ if and only if $h=g$, 
and the same holds for $s(h)$, so that $YU=UY$. 
Hence every element of $r^{-1}(Y)\cap s^{-1}(Y)$ is regular and 
$\cG|Y=r^{-1}(Y)\cap s^{-1}(Y)$. 
Its unit space $Y$ being finite and its range map being 
a local homeomorphism, $\cG|Y$ is discrete, 
so $\Iso(\cG|Y)^\circ=\Iso(\cG|Y)$ is closed and 
\eqref{eq:hausdorffreduction} holds. 
The effective reduction $\cG{\downarrow}Y$ is then 
the equivalence relation on $Y$ which identifies two points 
lying in the same $\cG$-orbit; 
in particular it is minimal if and only if 
$Y$ is contained in a single $\cG$-orbit. 
Since $Y$ is discrete, a group acts minimally on $Y$ 
if and only if it acts transitively, 
and we use the two terms interchangeably in this situation. 
Finally, a finite subset of $\cG^{(0)}$ is automatically nowhere 
dense, because $\cG^{(0)}$ has no isolated points. 

For an irreducible SFT groupoid 
(see \cite[Section 2.3]{Ma26maximal1}), 
the regular reduction to a closed set with 
no isolated points is effective, so the Hausdorffness condition 
also holds; see \cite[Lemma 5.1]{Ma26maximal3}.

We now relate these reductions to the action of the stabilizer on $Y$. 
Under the hypotheses below, every compact open full bisection of 
the regular reduction extends to an element of the stabilizer. 
To distinguish such extensions from their actions on $Y$, 
we temporarily keep track of the full bisections themselves. 
If $\cH$ is a Hausdorff ample groupoid with compact unit space, 
not necessarily effective, we write 
\[ 
\widetilde{\sF}(\cH) 
:=\{U\subset\cH\mid U\text{ is a compact open full bisection}\}. 
\] 
This is a group under bisection multiplication, 
with identity $\cH^{(0)}$ and inverse $U\mapsto U^{-1}$. 
For the empty groupoid, 
it is the trivial group consisting of the empty bisection. 
The assignment $U\mapsto\theta_U$ is 
a homomorphism into $\Homeo(\cH^{(0)})$. 
If $\cH$ is effective, it identifies $\widetilde{\sF}(\cH)$ with 
the topological full group $\sF(\cH)$ defined earlier. 
In general, different full bisections may induce the same homeomorphism. 
We use the notation $\widetilde{\sF}$ only in this subsection. 

Let $Y$ be a closed subset of $\cG^{(0)}$. 
We define a map $p:\St(\cG,Y)\to\widetilde{\sF}(\cG|Y)$ as follows. 
For $\alpha\in\St(\cG,Y)$, 
there exists a full bisection $U\in\cCO(\cG)$ 
such that $\alpha=\theta_U$. 
As $YU=UY$, any element $g\in YUY$ is regular, 
and so $YUY$ is a full bisection of $\cG|Y$. 
Set $p(\alpha):=YUY$. 

When $Y$ satisfies \eqref{eq:hausdorffreduction}, 
we define another map $q:\St(\cG,Y)\to\sF(\cG{\downarrow}Y)$ 
by letting $q(\alpha)$ be the restriction of $\alpha$ to $Y$. 

\begin{proposition}\label{structureofSt} 
Let $Y$ be a closed subset of $\cG^{(0)}$. 
Suppose that $\cG$ is minimal and either almost finite or purely infinite. 
\begin{enumerate} 
\item The map $p$ defined above is a well-defined surjective homomorphism 
from $\St(\cG,Y)$ to $\widetilde{\sF}(\cG|Y)$, and 
\[ 
\Ker p=\{\alpha\in\St(\cG,Y) 
\mid\text{$\alpha$ equals $\id$ on a neighborhood of $Y$}\}. 
\] 
\item Suppose in addition that $Y$ satisfies \eqref{eq:hausdorffreduction}. 
The map $q$ defined above is a well-defined surjective homomorphism 
from $\St(\cG,Y)$ to $\sF(\cG{\downarrow}Y)$, and 
\[ 
\Ker q=\{\alpha\in\St(\cG,Y) 
\mid\text{$\alpha$ equals $\id$ on $Y$}\}. 
\] 
\end{enumerate} 
If $Y\neq\cG^{(0)}$, the restrictions of $p$ and, when defined, $q$ 
to $\St(\sD(\cG),Y)$ are also surjective. 
\end{proposition} 

\begin{proof} 
(1)\: 
The full bisection implementing $\alpha\in\St(\cG,Y)$ is unique 
because $\cG$ is effective, so $p$ is well-defined. 
If $U,W$ implement elements of $\St(\cG,Y)$, then 
$Y(UW)Y=(YUY)(YWY)$, and hence $p$ is a homomorphism. 
Suppose that $\theta_U$ is in $\Ker p$. 
Then $YUY=Y$, and so we get $Y\subset U$. 
It follows that an open neighborhood of $Y$ is contained in $U$, 
which means $\theta_U$ equals the identity on a neighborhood of $Y$. 
Conversely, 
if $\theta_U$ equals the identity on an open neighborhood $A$ of $Y$, 
then $UA=A$ by effectiveness of $\cG$, and hence $YUY=Y$. 

It remains to show that $p$ is surjective. 
The cases $Y=\emptyset$ and $Y=\cG^{(0)}$ are immediate. 
We may therefore assume that $Y$ is nonempty and proper. 
Let $V\in\cCO(\cG|Y)$ be a full bisection. 
For every $g\in V$, there exists a bisection $V_g\in\cCO(\cG)$ 
such that $g\in V_g$, $YV_g=V_gY$ and $YV_gY\subset V$. 
The last condition can be obtained by shrinking $V_g$, 
since $V$ is open in $\cG|Y$ with its relative topology. 
By compactness of $V\subset\cG$, 
we may find finitely many bisections $V_1,V_2,\dots,V_m\in\cCO(\cG)$ 
which, after taking a disjoint compact open refinement, 
are mutually disjoint, cover $V$ 
and satisfy $YV_i=V_iY$ and $YV_iY\subset V$. 
Let $(A_n)_n$ be a decreasing sequence of 
proper clopen subsets of $\cG^{(0)}$ such that $\bigcap_nA_n=Y$. 
Then, for any $n$, 
the bisections $A_nV_iA_n$, $i=1,2,\dots,m$ are mutually disjoint 
and still cover $V$. 
We claim that for sufficiently large $n$, 
\[ 
r(A_nV_iA_n),\ i=1,2,\dots,m 
\] 
are mutually disjoint. 
Otherwise, after passing to a subsequence and relabeling, 
we may fix $i\neq j$ and choose $g_n\in A_nV_iA_n$ and $h_n\in A_nV_jA_n$ 
such that $r(g_n)=r(h_n)$. 
We may assume that 
$g_n$ and $h_n$ converge to $g\in YV_iY$ and $h\in YV_jY$, respectively. 
One has $r(g)=r(h)$. 
Since $V$ is a bisection, we get $g=h$, which is a contradiction. 
In the same way, 
\[ 
s(A_nV_iA_n),\ i=1,2,\dots,m 
\] 
are mutually disjoint for sufficiently large $n$. 
For a large enough $n$, we set 
\[ 
U:=A_nV_1A_n\sqcup A_nV_2A_n\sqcup\dots\sqcup A_nV_mA_n. 
\] 
Then $U\in\cCO(\cG)$ is a bisection satisfying $YUY=V$. 
Moreover, $r(U),s(U)\subset A_n\neq\cG^{(0)}$. 
Using Lemma~\ref{elementofDG}, 
we may extend $U$ to a full bisection $W$ with $\theta_W\in\sD(\cG)$. 
Since $Y=s(V)\subset s(U)$ and $Y=r(V)\subset r(U)$, 
and $W$ is a bisection containing $U$, both $s(W\setminus U)$ 
and $r(W\setminus U)$ are disjoint from $Y$. 
Thus $\theta_W\in\St(\sD(\cG),Y)$ and $p(\theta_W)=YWY=V$, 
proving surjectivity even on $\St(\sD(\cG),Y)$. 

(2)\: 
Put $\cH=\cG|Y$ and let $\pi:\cH\to\Eff(\cH)$ be the quotient map, 
whose properties were recorded after Definition~\ref{def:effectivequotient}. 
By \eqref{eq:hausdorffreduction}, 
the groupoid $\Eff(\cH)=\cG{\downarrow}Y$ is an effective ample groupoid, 
so that $\sF(\Eff(\cH))$ is defined 
and may be identified with $\widetilde{\sF}(\Eff(\cH))$ as above. 
For $U\in\widetilde{\sF}(\cH)$, 
the image $\pi(U)$ is a compact open full bisection of $\Eff(\cH)$. 
Thus $U\mapsto\pi(U)$ is a homomorphism 
$\widetilde{\sF}(\cH)\to\sF(\Eff(\cH))$, and it is surjective. 
Indeed, let $B$ be a compact open full bisection of $\Eff(\cH)$. 
Each arrow of $B$ has a lift to a compact open bisection of $\cH$ 
whose image under $\pi$ is contained in $B$. 
Choose finitely many such bisections $U_1,\dots,U_k$ 
whose source sets cover $Y$. 
Choose a clopen partition $Y=\bigsqcup_{j=1}^k C_j$ 
with $C_j\subset s(U_j)$, allowing empty pieces. 
Then $U:=\bigsqcup_j U_jC_j$ has source $Y$. 
Its range sets are disjoint and cover $Y$ 
because their images are the corresponding range sets 
of the full bisection $B$. 
Hence $U$ is a full bisection of $\cH$ and $\pi(U)=B$. 
Composing this surjection with $p$, we obtain $q$: 
if a full bisection $W$ of $\cG$ implements $\alpha\in\St(\cG,Y)$, 
then $\theta_{\pi(YWY)}=\alpha|Y$. 
Therefore $q$ is a well-defined surjective homomorphism. 
If $Y\neq\cG^{(0)}$, (1) also gives surjectivity of $q$ on 
$\St(\sD(\cG),Y)$. 
It is evident that $\Ker q$ equals 
$\{\alpha\in\St(\cG,Y)\mid\text{$\alpha$ equals $\id$ on $Y$}\}$. 
\end{proof} 

The orbit comparison needed below is local and does not require 
almost finiteness or pure infiniteness. 

\begin{proposition}\label{stabilizerminimality} 
Let $\cG$ be a minimal ample groupoid with Cantor unit space $X$, 
let $Y\subset X$ be a nonempty closed subset, and let $\Gamma$ be a group 
with $\sD(\cG)\subset\Gamma\subset\sF(\cG)$. 
Then the action of $\St(\Gamma,Y)$ on $Y$ has the same orbits 
as the regular reduction $\cG|Y$. 
In particular, $\St(\Gamma,Y)$ acts minimally on $Y$ 
if and only if $\cG|Y$ is minimal. 
When $Y$ satisfies \eqref{eq:hausdorffreduction}, 
these conditions are also equivalent to minimality 
of $\cG{\downarrow}Y$. 
If these conditions hold and $Y$ is infinite, then $Y$ is a Cantor set. 
\end{proposition} 

\begin{proof} 
Every element of $\St(\Gamma,Y)$ restricts to a homeomorphism 
of $Y$ induced by a full bisection of $\cG|Y$. 
Thus each $\St(\Gamma,Y)$-orbit is contained in a $\cG|Y$-orbit. 
If $Y=X$, the reverse inclusion follows from 
Lemma~\ref{threecycle} (1)--(2), by completing each arrow 
with distinct source and range to a three-cycle. 
Suppose that $Y$ is proper, and let $g\in\cG|Y$ 
with $x:=s(g)\neq r(g)$. 
Choose a compact open bisection $U$ containing $g$ with $YU=UY$. 
By minimality of $\cG$, the orbit of $x$ contains distinct points 
$z_1,z_2\in X\setminus Y$. 
Choose compact open bisections $W_1,W_2$ with 
$x\in s(W_j)$ and $\theta_{W_j}(x)=z_j$. 
Shrinking these bisections and $U$, we may assume that 
\[ 
s(W_1)=s(W_2)=s(U), 
\] 
that $s(U)$, $r(U)$, $r(W_1)$ and $r(W_2)$ are mutually disjoint, 
and that $r(W_1)\cup r(W_2)$ is disjoint from $Y$. 
Put $V:=W_2W_1^{-1}$. 
Lemma~\ref{involution}, with $Z=W_1$, gives 
$\alpha:=\tau_U\tau_V\in\sD(\cG)$. 
The transposition $\tau_U$ preserves $Y$, and $\tau_V$ fixes $Y$ 
pointwise because its support is disjoint from $Y$. 
Hence $\alpha\in\St(\Gamma,Y)$ and $\alpha(x)=r(g)$. 
This proves the orbit equality. 

The quotient map from $\cG|Y$ to its effective quotient 
does not change orbits, so the minimality assertions follow. 
Finally, if $Y$ is infinite and the stabilizer action is minimal, 
the nonisolated points of $Y$ form a nonempty closed invariant subset. 
They therefore exhaust $Y$, which is consequently a Cantor set. 
\end{proof}

\subsection{Local subgroups and Property E}\label{subsec:propertyE} 

Throughout this subsection, 
we let $\cG$ be a minimal ample groupoid with Cantor unit space 
and assume that $\cG$ is either almost finite or purely infinite. 

We use the local subgroup notation of Definition~\ref{def:localsubgroup}. 

\begin{definition}\label{def:propertyE} 
Let $\Gamma$ be a subgroup of $\sF(\cG)$. 
Following \cite[Section 5]{GV26ETDS}, 
we say that $\Gamma$ has Property E if 
$\Gamma_{V_1\cup V_2}=\left\langle\Gamma_{V_1}\cup\Gamma_{V_2}\right\rangle$ 
holds whenever $V_1,V_2\subset\cG^{(0)}$ are clopen sets 
satisfying $V_1\cap V_2\neq\emptyset$. 
\end{definition} 

In this subsection and the next, 
we work with the following class of groups. 

\begin{setting}\label{intermediategroup} 
Let $\Gamma$ be a group with $\sD(\cG)\subset\Gamma\subset\sF(\cG)$. 
Since the quotient $\sF(\cG)/\sD(\cG)$ is abelian, 
every such $\Gamma$ is a normal subgroup of $\sF(\cG)$. 
Conversely, by Theorem~\ref{niceproperties} (1), 
every non-trivial normal subgroup of $\sF(\cG)$ contains $\sD(\cG)$. 
Thus the groups considered here are precisely 
the non-trivial normal subgroups of $\sF(\cG)$. 
The choices $\Gamma=\sF(\cG)$ and $\Gamma=\sD(\cG)$ 
are the two extreme cases. 
Another case of interest is $\Gamma=\Ker I$, 
where $I:\sF(\cG)\to H_1(\cG)$ is the index map. 
\end{setting} 

We now recall why $\sF(\cG)$ itself has Property E. 
The group $\sF(\cG)$ is fully amplified in the sense of 
\cite[Definition 2.5]{GV26ETDS}, that is, it coincides with its own 
full amplification, and it acts minimally on $\cG^{(0)}$. 
The cancellation property in Theorem~\ref{niceproperties} (2) implies 
that $\sF(\cG)$ has Property HS of \cite[Section 5]{GV26ETDS}. 
Indeed, let $A,B\subset\cG^{(0)}$ be nonempty proper clopen sets. 
Suppose that $1_A$ and $1_B$ represent the same class in the 
coinvariant group $C(\cG^{(0)},\Z)_{\sF(\cG)}$ for the natural action 
on $C(\cG^{(0)},\Z)$. 
Relations induced by elements of $\sF(\cG)$ are groupoid boundaries, 
so $[1_A]=[1_B]$ in $H_0(\cG)$. 
The same equality holds for their complements. 
Applying cancellation to $A,B$ and to 
$\cG^{(0)}\setminus A,\cG^{(0)}\setminus B$, 
and taking the union of the resulting bisections, 
we obtain $\alpha\in\sF(\cG)$ with $\alpha(A)=B$. 
This is precisely Property HS. 
By \cite[Proposition 8.6]{GV26ETDS}, $\sF(\cG)$ therefore has Property E. 

To pass from $\sF(\cG)$ to $\Gamma$, we use 
Lemma~\ref{localabelianization} to represent every abelianization class 
inside any nonempty clopen set. 
This lets us adjust a factorization in $\sF(\cG)$ to one in $\Gamma$, 
using corrections supported in the overlap of the two clopen sets. 

\begin{proposition}\label{propertyE} 
Let $\Gamma$ be as in Setting~\ref{intermediategroup}. 
Then $\Gamma$ has Property E. 
\end{proposition} 

\begin{proof} 
Let $V_1,V_2\subset\cG^{(0)}$ be clopen sets 
with $W:=V_1\cap V_2\neq\emptyset$ and 
let $\alpha\in\Gamma_{V_1\cup V_2}$. 
We must write $\alpha$ as a product of elements of 
$\Gamma_{V_1}\cup\Gamma_{V_2}$. 

Let $\pi_\Gamma:\sF(\cG)\to\sF(\cG)/\Gamma$ be the quotient map, 
which is well defined because $\Gamma$ is normal in $\sF(\cG)$, 
and write the abelian group $\sF(\cG)/\Gamma$ additively. 
As $\pi_\Gamma$ factors through $\pi:\sF(\cG)\to H_1(\sF(\cG))$, 
Lemma~\ref{localabelianization} gives 
$\pi_\Gamma(\sF(\cG)_W)=\sF(\cG)/\Gamma$. 

Since $\sF(\cG)$ has Property E, we may write 
$\alpha=\beta_1\beta_2\cdots\beta_n$ 
with $\beta_i\in\sF(\cG)_{V_{j(i)}}$ for some $j(i)\in\{1,2\}$. 
For each $i$, choose $\delta_i\in\sF(\cG)_W$ with 
$\pi_\Gamma(\delta_i)=\pi_\Gamma(\beta_i)$ and 
put $\gamma_i:=\beta_i\delta_i^{-1}$. 
Then $\pi_\Gamma(\gamma_i)=0$, that is, $\gamma_i\in\Gamma$, and 
\[ 
\supp(\gamma_i)\subset\supp(\beta_i)\cup\supp(\delta_i) 
\subset V_{j(i)}\cup W=V_{j(i)}, 
\] 
so that $\gamma_i\in\Gamma_{V_{j(i)}}$. 

Put $d_0:=1$ and $d_i:=\delta_1\delta_2\cdots\delta_i$ for $i\geq1$. 
From $\beta_i=\gamma_i\delta_i$ we obtain 
\[ 
\alpha=\gamma_1\cdot(d_1\gamma_2d_1^{-1})\cdot(d_2\gamma_3d_2^{-1}) 
\cdots(d_{n-1}\gamma_nd_{n-1}^{-1})\cdot d_n. 
\] 
For each $i$ we have $\supp(d_{i-1})\subset W\subset V_{j(i)}$, 
whence $d_{i-1}(V_{j(i)})=V_{j(i)}$ and 
\[ 
\supp(d_{i-1}\gamma_id_{i-1}^{-1})=d_{i-1}(\supp(\gamma_i)) 
\subset V_{j(i)}. 
\] 
As $\Gamma$ is normal in $\sF(\cG)$, 
each $d_{i-1}\gamma_id_{i-1}^{-1}$ lies in $\Gamma_{V_{j(i)}}$. 
Finally, 
\[ 
\pi_\Gamma(d_n)=\sum_{i=1}^n\pi_\Gamma(\delta_i) 
=\sum_{i=1}^n\pi_\Gamma(\beta_i)=\pi_\Gamma(\alpha)=0, 
\] 
so that $d_n\in\Gamma_W\subset\Gamma_{V_1}$. 
This completes the proof. 
\end{proof}

\subsection{Maximality of stabilizers}\label{subsec:maximality} 

Throughout this subsection, 
we let $\cG$ be a minimal ample groupoid with Cantor unit space 
and assume that $\cG$ is either almost finite or purely infinite. 
We let $\Gamma$ be as in Setting~\ref{intermediategroup}. 

Using Property E, we first characterize maximality by the action 
of the stabilizer on $Y$. 
We then express this criterion in terms of the effective reduction. 

\begin{theorem}\label{stabilizeractionmaximal} 
Let $\cG$ be a minimal ample groupoid with Cantor unit space, 
and suppose that $\cG$ is either almost finite or purely infinite. 
Let $\Gamma$ be as in Setting~\ref{intermediategroup}. 
Let $Y\subset\cG^{(0)}$ be a nonempty closed set 
that is nowhere dense in $\cG^{(0)}$. 
Then $\St(\Gamma,Y)$ is a maximal subgroup of $\Gamma$ 
if and only if it acts minimally on $Y$. 
\end{theorem} 

\begin{proof} 
For a clopen subset $V\subset\cG^{(0)}$ we have 
\begin{equation}\label{eq:localstab} 
V\cap Y=\emptyset\implies\Gamma_V\subset\St(\Gamma,Y), 
\end{equation} 
because every $\alpha\in\Gamma_V$ is the identity on $Y$. 
We also note that $\sD(\cG)\subset\Gamma$ acts minimally on $\cG^{(0)}$ 
by Lemma~\ref{threecycle} (3); 
in particular, no nonempty proper closed subset of $\cG^{(0)}$ 
is invariant under $\Gamma$. 

Suppose first that $\St(\Gamma,Y)$ acts minimally on $Y$. 
Since $Y$ is a nonempty proper closed subset of $\cG^{(0)}$, 
it is not invariant under $\Gamma$, and so 
$\St(\Gamma,Y)$ is a proper subgroup of $\Gamma$. 
Let $K$ be a subgroup of $\Gamma$ 
with $\St(\Gamma,Y)\subsetneq K\subset\Gamma$ and put 
\[ 
L:=\{V\subset\cG^{(0)}\mid 
\text{$V$ is clopen and }\Gamma_V\subset K\}. 
\] 
The family $L$ has the following three properties. 
\begin{itemize} 
\item If $V\in L$ and $V'\subset V$ is clopen, then $V'\in L$, 
because $\Gamma_{V'}\subset\Gamma_V$. 
\item If $V_1,V_2\in L$ satisfy $V_1\cap V_2\neq\emptyset$, 
then $V_1\cup V_2\in L$ by Proposition~\ref{propertyE}. 
\item If $V\in L$ and $\alpha\in K$, then $\alpha(V)\in L$, 
because $\Gamma_{\alpha(V)}=\alpha\Gamma_V\alpha^{-1}\subset K$. 
\end{itemize} 
By \eqref{eq:localstab}, 
every clopen set disjoint from $Y$ belongs to $L$. 

Choose $\gamma\in K\setminus\St(\Gamma,Y)$. 
Then $\gamma(Y)\neq Y$, and after replacing $\gamma$ by $\gamma^{-1}$ 
if necessary, we may assume that 
there exists $y_0\in Y$ with $\gamma(y_0)\notin Y$. 
By continuity there is a clopen neighborhood $A$ of $y_0$ 
such that $\gamma(A)\cap Y=\emptyset$. 
Then $\gamma(A)\in L$, and applying the third property 
to $\gamma^{-1}\in K$ we obtain $A\in L$. 
Let $x\in Y$. 
Since $A\cap Y$ is a nonempty relatively open subset of $Y$ and 
$\St(\Gamma,Y)$ acts minimally on $Y$, 
there exists $\alpha\in\St(\Gamma,Y)\subset K$ with $\alpha(x)\in A$. 
Hence $\alpha^{-1}(A)$ is a clopen neighborhood of $x$ 
belonging to $L$. 
As $Y$ is compact, we obtain $W_1,W_2,\dots,W_p\in L$ 
such that $Y\subset W_1\cup W_2\cup\dots\cup W_p$. 

Since $Y$ is nowhere dense, 
$W_i\setminus Y$ is a nonempty open set for each $i$, 
and so we may choose a nonempty clopen set $D_i\subset W_i\setminus Y$. 
Put $D:=D_1\cup D_2\cup\dots\cup D_p$. 
Then $D\cap Y=\emptyset$ and hence $D\in L$. 
Since 
\[ 
\left(D\cup W_1\cup\dots\cup W_{i-1}\right)\cap W_i\supset D_i 
\neq\emptyset 
\] 
for every $i$, repeated use of the second property yields 
\[ 
N:=D\cup W_1\cup W_2\cup\dots\cup W_p\in L. 
\] 
Note that $Y\subset N$. 
Put $E:=D_1\cup(\cG^{(0)}\setminus N)$. 
Then $E$ is clopen and disjoint from $Y$, so that $E\in L$, 
and $N\cap E\supset D_1\neq\emptyset$. 
The second property therefore gives 
$\cG^{(0)}=N\cup E\in L$, that is, $\Gamma=\Gamma_{\cG^{(0)}}\subset K$. 
Thus $K=\Gamma$, and $\St(\Gamma,Y)$ is a maximal subgroup of $\Gamma$. 

Suppose next that $\St(\Gamma,Y)$ does not act minimally on $Y$. 
Then there exists a nonempty closed subset $Z\subsetneq Y$ 
which is invariant under $\St(\Gamma,Y)$. 
Every element of $\St(\Gamma,Y)$ preserves $Z$, whence 
\[ 
\St(\Gamma,Y)\subset\St(\Gamma,Z)\subset\Gamma. 
\] 
Since $Z$ is a nonempty proper closed subset of $\cG^{(0)}$, 
it is not invariant under $\Gamma$, 
and therefore $\St(\Gamma,Z)\neq\Gamma$. 
It remains to see that $\St(\Gamma,Y)\neq\St(\Gamma,Z)$. 
Choose $y\in Y\setminus Z$ and 
a clopen neighborhood $A$ of $y$ with $A\cap Z=\emptyset$. 
As $Y$ is nowhere dense, 
we may choose a nonempty clopen set $B\subset A\setminus Y$. 
By Lemma~\ref{threecycle} (3), 
the local subgroup $\sD(\cG)_A$ acts minimally on $A$, 
so there exists $\sigma\in\sD(\cG)_A\subset\Gamma$ with $\sigma(y)\in B$. 
Since $\supp(\sigma)\subset A$ and $A\cap Z=\emptyset$, 
the element $\sigma$ is the identity on $Z$, 
and hence $\sigma\in\St(\Gamma,Z)$. 
On the other hand $y\in Y$ and $\sigma(y)\in B\subset\cG^{(0)}\setminus Y$, 
so $\sigma(Y)\neq Y$ and $\sigma\notin\St(\Gamma,Y)$. 
Consequently $\St(\Gamma,Y)$ is not a maximal subgroup of $\Gamma$. 
\end{proof}

We can now state the groupoid form of the maximality criterion. 

\begin{theorem}\label{stabilizermaximal} 
Let $\cG$, $\Gamma$ and $Y$ be as in Theorem~\ref{stabilizeractionmaximal}, 
and suppose in addition that $Y$ satisfies \eqref{eq:hausdorffreduction}. 
Then $\St(\Gamma,Y)$ is a maximal subgroup of $\Gamma$ 
if and only if the ample groupoid $\cG{\downarrow}Y$ is minimal. 
\end{theorem} 

\begin{proof} 
By Proposition~\ref{stabilizerminimality}, $\St(\Gamma,Y)$ acts 
minimally on $Y$ if and only if $\cG{\downarrow}Y$ is minimal. 
The conclusion therefore follows from 
Theorem~\ref{stabilizeractionmaximal}. 
\end{proof} 

\begin{corollary}\label{stabilizermaximalFD} 
Let $\cG$ and $Y$ be as in Theorem~\ref{stabilizermaximal}. 
The following conditions are equivalent. 
\begin{enumerate} 
\item The ample groupoid $\cG{\downarrow}Y$ is minimal. 
\item $\St(\cG,Y)$ is a maximal subgroup of $\sF(\cG)$. 
\item $\St(\cG,Y)\cap\sD(\cG)$ is a maximal subgroup of $\sD(\cG)$. 
\item $\St(\cG,Y)\cap\Ker I$ is a maximal subgroup of $\Ker I$, 
where $I:\sF(\cG)\to H_1(\cG)$ is the index map. 
\end{enumerate} 
\end{corollary} 

\begin{proof} 
Apply Theorem~\ref{stabilizermaximal} to 
$\Gamma=\sF(\cG)$, $\Gamma=\sD(\cG)$ and $\Gamma=\Ker I$. 
\end{proof} 

\begin{remark} 
For $\Gamma=\sF(\cG)$ and infinite $Y$, 
Theorem~\ref{stabilizeractionmaximal} agrees with 
\cite[Theorem 1.3]{GV26ETDS}; for finite $Y$, it agrees with 
\cite[Theorem 1.2]{GV26ETDS}. 
Our proof uses Property E and local minimality to treat both cases 
and extend the criterion to the intermediate groups considered here. 
Theorem~\ref{stabilizeractionmaximal} itself does not require 
the effective reduction to be Hausdorff. 

The hypothesis that the group is fully amplified prevents the direct 
application of these results when $\Gamma\subsetneq\sF(\cG)$. 
Indeed, \cite[Lemma 3.7]{Ma15crelle} shows that every arrow of $\cG$ 
is contained in a full bisection inducing an element of $\sD(\cG)$. 
Thus the full amplification of every group 
$\sD(\cG)\subset\Gamma\subset\sF(\cG)$ is $\sF(\cG)$, 
whereas a fully amplified group must coincide with its full 
amplification. 
\end{remark} 

\begin{remark}\label{rem:stabilizerabelianization} 
The stabilizer also maps onto the abelianization of the ambient 
topological full group. 
Thus the surjectivity condition appearing in 
\cite[Theorem 3.13]{Ma26maximal1} is automatic for stabilizers. 
Indeed, choose a nonempty clopen set $W\subset\cG^{(0)}\setminus Y$. 
Then $\sF(\cG)_W\subset\St(\cG,Y)$, and Lemma~\ref{localabelianization} 
gives 
\[ 
\St(\cG,Y)\sD(\cG)=\sF(\cG),\qquad 
\St(\cG,Y)/\St(\sD(\cG),Y)\cong H_1(\sF(\cG)). 
\] 
This holds for every proper closed set $Y$, without assuming maximality. 
\end{remark}

\begin{remark} 
Theorem~\ref{stabilizeractionmaximal} includes the case of a finite $Y$, 
where minimality of the action means transitivity. 
When $Y=\{y\}$ is a singleton this condition is vacuous, 
so that $\St(\Gamma,\{y\})$ is always a maximal subgroup of $\Gamma$. 
Equivalently, the action of $\Gamma$ on the orbit $\Gamma y$ is primitive; 
this is the usual correspondence between the subgroups of $\Gamma$ 
containing a point stabilizer and the blocks containing that point. 

For a singleton $Y$ neither Property E nor the almost finite 
or purely infinite hypothesis is needed. 
Let $\cG$ be a minimal ample groupoid with $\cG^{(0)}$ a Cantor set and 
let $\Gamma$ satisfy $\sD(\cG)\subset\Gamma\subset\sF(\cG)$. 
By Lemma~\ref{threecycle} (3) the orbit $\Gamma y$ is infinite, 
so $\St(\Gamma,\{y\})$ is a proper subgroup of $\Gamma$. 
Let $K$ be a subgroup with $\St(\Gamma,\{y\})\subsetneq K\subset\Gamma$ 
and choose $\gamma\in K$ with $y':=\gamma(y)\neq y$. 
Let $z\in\Gamma y\setminus\{y,y'\}$ and choose a clopen set $A$ 
with $y',z\in A$ and $y\notin A$. 
The points $y'$ and $z$ lie in a common $\cG|A$-orbit, 
so the proof of Lemma~\ref{threecycle} (3), applied to an arrow 
from $y'$ to $z$ and a clopen neighborhood of $z$ in $A$ 
not containing $y'$, produces 
$\sigma\in\sD(\cG|A)\subset\sD(\cG)_A$ with $\sigma(y')=z$. 
As $\supp(\sigma)\subset A$ and $y\notin A$, 
we have $\sigma\in\St(\Gamma,\{y\})\subset K$, 
and hence $\sigma\gamma\in K$ sends $y$ to $z$. 
Since $y,y'\in Ky$, we conclude that $Ky=\Gamma y$. 
Given $\alpha\in\Gamma$, choose $\kappa\in K$ with $\kappa(y)=\alpha(y)$; 
then $\kappa^{-1}\alpha\in\St(\Gamma,\{y\})\subset K$, so that $\alpha\in K$. 
Thus $K=\Gamma$. 
\end{remark}

\begin{example}\label{ex:AFnonetalereduction} 
Let $\mathcal R$ be the tail equivalence groupoid on 
$X:=\{0,1\}^{\N}$ with its usual AF topology, and put 
\[ 
Y:=\left\{w_1w_2\cdots\mid w_j\in\{00,1\}\text{ for all }j\in\N\right\}
\subset X. 
\] 
Thus $\mathcal R$ is the minimal UHF groupoid of type $2^\infty$, 
whose basic bisections replace finite prefixes of equal length. 
Concatenation of the words $00$ and $1$ identifies $Y$ with a 
two-symbol Cantor space. 
Moreover, every prefix occurring in $Y$ can be extended to a finite 
concatenation of these words, and then by $01$ to a cylinder 
disjoint from $Y$. 
Hence $Y$ is nowhere dense. 

The set $Y$ is not $\mathcal R$-\'etale. 
Indeed, the prefix replacement $0z\mapsto 1z$ gives an arrow 
from $0^\infty$ to $10^\infty$, both of which belong to $Y$. 
This arrow is not regular, since 
\[ 
0^{2k}1^\infty\in Y,\qquad 
10^{2k-1}1^\infty\notin Y 
\] 
for $k\geq1$, and $0^{2k}1^\infty\to0^\infty$ as $k\to\infty$. 

Nevertheless, the regular reduction $\mathcal R|Y$ is minimal. 
To describe it, use a state $A$ at the end of each complete code word 
and a state $B$ after the first symbol of $00$. 
Starting at $A$, the allowed transitions are 
\[ 
A\xrightarrow{1}A,\qquad 
A\xrightarrow{0}B,\qquad 
B\xrightarrow{0}A. 
\] 
For an admissible finite prefix $\mu$, the set of continuations 
$\{z\in X\mid\mu z\in Y\}$ is $Y$ or $0Y$, according as 
the terminal state is $A$ or $B$. 
These sets intersect only at $0^\infty$ and do not agree on any 
neighborhood of that point. 
Consequently, a prefix replacement between two admissible prefixes 
of equal length is regular at an arrow with endpoints in $Y$ 
if and only if their terminal states agree. 
The label map therefore identifies $\mathcal R|Y$ with the AF 
groupoid of the stationary Bratteli diagram rooted at $A$ 
with incidence matrix 
\[ 
M=\begin{bmatrix}1&1\\1&0\end{bmatrix}. 
\] 
Since every entry of $M^2$ is positive, this AF groupoid is minimal. 
It is principal, so $\mathcal R{\downarrow}Y=\mathcal R|Y$. 
Theorem~\ref{stabilizermaximal} now shows that 
$\St(\mathcal R,Y)$ is a maximal subgroup of 
$\sD(\mathcal R)=\sF(\mathcal R)$. 

The points $0^\infty$ and $10^\infty$ belong to the same 
$\mathcal R$-orbit, but their state paths remain distinct at 
every positive level. 
Thus they lie in different $\mathcal R|Y$-orbits, and no element 
of $\St(\mathcal R,Y)$ sends one to the other. 
This illustrates why the regular reduction, rather than the full 
restriction, records the orbits of the stabilizer on $Y$. 
\end{example}

The local preservation condition used to define regular reductions 
also yields a wide open subgroupoid on the whole unit space, 
whose topological full group is the stabilizer. 
This construction does not require minimality, almost finiteness or 
pure infiniteness. 

\begin{definition}\label{def:preservingsubgroupoid} 
Let $\cG$ be an ample groupoid with Cantor unit space and 
let $Y\subset\cG^{(0)}$ be a closed subset. 
We call 
\[ 
\cS_Y:=\bigcup\left\{U\in\cCO(\cG)\mid 
\text{$U$ is a bisection and }YU=UY\right\} 
\] 
the $Y$-preserving subgroupoid of $\cG$. 
\end{definition} 

The defining bisections are closed under inverses and nonempty 
products and cover the unit space. 
Hence $\cS_Y$ is a wide open subgroupoid of $\cG$. 
An arrow with both endpoints outside $Y$ lies in a bisection whose 
source and range miss $Y$, so the definition of regularity gives 
\[ 
\cS_Y=\cG|Y\sqcup\cG|(\cG^{(0)}\setminus Y) 
\] 
as sets. 
In particular, $Y$ is $\cS_Y$-invariant, and with the relative 
topologies, 
\[ 
\cS_Y|Y=\cG|Y,\qquad 
\cS_Y|(\cG^{(0)}\setminus Y)=\cG|(\cG^{(0)}\setminus Y). 
\] 
Moreover, a full bisection of $\cG$ preserves $Y$ precisely when it is 
contained in $\cS_Y$, and hence $\sF(\cS_Y)=\St(\cG,Y)$. 

Recall that a wide open subgroupoid $\cK\subset\cG$ is prime if 
there is no wide open subgroupoid $\cH$ with 
$\cK\subsetneq\cH\subsetneq\cG$ \cite[Definition 5.1]{Ma26maximal1}. 
The first part of the next proposition requires only minimality of 
$\cG$; the almost finite or purely infinite hypothesis is needed 
only for the final equivalence. 

\begin{proposition}\label{prop:preservingsubgroupoid} 
Let $\cG$ be a minimal ample groupoid with Cantor unit space, 
and let $Y\subset\cG^{(0)}$ be a nonempty proper closed set. 
Then the wide open subgroupoid $\cS_Y$ is proper in $\cG$, and 
\[ 
\cS_Y\subset\cG\text{ is prime} 
\iff \cG|Y\text{ is minimal}. 
\] 
If these conditions hold and $Y$ is nowhere dense, then $Y$ is the 
unique nonempty proper closed $\cS_Y$-invariant subset of 
$\cG^{(0)}$, and every $\cS_Y$-orbit in $\cG^{(0)}\setminus Y$ 
is dense in $\cG^{(0)}$. Moreover, $Y$ is either finite or a Cantor set. 

Suppose in addition that $\cG$ is either almost finite or purely 
infinite and that $Y$ is nowhere dense. Then, for every 
$\Gamma$ as in Setting~\ref{intermediategroup}, 
\[ 
\cS_Y\subset\cG\text{ is prime} 
\iff \St(\Gamma,Y)\text{ is maximal in }\Gamma. 
\] 
\end{proposition} 

\begin{proof} 
The subgroupoid $\cS_Y$ is proper because $Y$ is a nonempty proper 
closed $\cS_Y$-invariant subset and $\cG$ is minimal. 
Suppose first that $\cG|Y$ is minimal, and let 
$\cS_Y\subsetneq\cK\subset\cG$ be a wide open subgroupoid. 
If $Y$ were $\cK$-invariant, every compact open bisection of $\cK$ 
would preserve $Y$, giving $\cK\subset\cS_Y$, a contradiction. 
Put $U:=\cG^{(0)}\setminus Y$ and $O:=r\left(s^{-1}(U)\cap\cK\right)$. 
Then $O$ is open and $\cK$-invariant, contains $U$, and meets $Y$. 
The set $Y\setminus O$ is therefore a proper closed 
$\cG|Y$-invariant subset of $Y$, so minimality gives $O=\cG^{(0)}$. 
For any $g\in\cG$, choose $k_r,k_s\in\cK$ carrying $r(g)$ and 
$s(g)$, respectively, into $U$. Then 
\[ 
k_rgk_s^{-1}\in\cG|U=\cS_Y|U\subset\cK, 
\] 
and hence $g\in\cK$. Thus $\cK=\cG$, proving primeness. 

Conversely, suppose that $\cG|Y$ is not minimal. 
Choose a nonempty proper closed $\cG|Y$-invariant subset $Z\subset Y$. 
Every bisection preserving $Y$ preserves $Z$, and hence 
$\cS_Y\subset\cS_Z$. The inclusion is strict. Indeed, choose 
$y\in Y\setminus Z$. By minimality of $\cG$, there is an arrow with 
source $y$ and range in $\cG^{(0)}\setminus Y$. Shrinking a compact 
open bisection around it, we obtain a bisection whose source and range 
are disjoint from $Z$ but which does not preserve $Y$. It belongs to 
$\cS_Z\setminus\cS_Y$. Since $Z$ is nonempty and proper, 
$\cS_Z\neq\cG$ by minimality. Hence $\cS_Y$ is not prime. 

Assume now that $Y$ is nowhere dense and $\cG|Y$ is minimal. 
For $x\in\cG^{(0)}\setminus Y$, its $\cS_Y$-orbit is the intersection 
of its $\cG$-orbit with $\cG^{(0)}\setminus Y$, and is therefore dense 
in $\cG^{(0)}$. Hence every proper closed $\cS_Y$-invariant subset is 
contained in $Y$, and minimality of $\cG|Y$ gives the asserted 
uniqueness. If $Y$ is infinite, its nonisolated points form a nonempty 
closed $\cG|Y$-invariant subset, and therefore equal $Y$. Thus $Y$ is 
a Cantor set. 

Finally, Proposition~\ref{stabilizerminimality} shows that 
$\cG|Y$ is minimal if and only if 
$\St(\Gamma,Y)$ acts minimally on $Y$. 
The last equivalence therefore follows from 
Theorem~\ref{stabilizeractionmaximal}. 
\end{proof}

\section{Isomorphism invariants and realizations} 
\label{sec:stabilizerinvariants} 

Throughout this section, $\cG$ is an effective ample groupoid 
with Cantor unit space. 
We study the isomorphism types of the stabilizers considered in 
Section~\ref{sec:stabilizer}. 
Section~\ref{subsec:spatial} proves spatial realization and 
identifies groupoid invariants determined by these groups. 
Section~\ref{subsec:effectiverealization} then shows that every 
AF groupoid with compact unit space is strongly realizable in 
any minimal ample groupoid, as a full restriction to a closed 
nowhere dense $\cG$-\'etale subset. 
The same realization theorem yields $2^{\aleph_0}$ pairwise 
nonisomorphic stabilizers in every group between 
$\sD(\cG)$ and $\sF(\cG)$. 
If the ambient groupoid is almost finite or purely infinite, 
these stabilizers are maximal subgroups. 
At the $C^*$-algebra level, the realization theorem gives explicit AF 
subquotients compatible with the canonical diagonal subalgebras, 
including a concrete CAR subquotient as in Glimm's theorem.

\subsection{Spatial realization for stabilizers}\label{subsec:spatial} 

We show that isomorphisms between the stabilizer groups 
considered here are induced by homeomorphisms of 
the ambient unit spaces. 
No almost finite or purely infinite hypothesis is imposed 
in this subsection. 
We first identify the effective reduction with the groupoid 
of germs of the restricted stabilizer action. 

\begin{lemma}\label{stabilizergerms} 
Let $\cG$ be an ample groupoid with Cantor unit space, 
let $\Gamma$ be a group with $\sD(\cG)\subset\Gamma\subset\sF(\cG)$, 
and let $Y\subset\cG^{(0)}$ be an infinite closed subset 
satisfying \eqref{eq:hausdorffreduction}. 
Suppose that the regular reduction $\cG|Y$ is minimal. 
Then $\cG{\downarrow}Y$ is canonically isomorphic to the groupoid 
of germs of the action of $\St(\Gamma,Y)$ on $Y$. 
\end{lemma} 

\begin{proof} 
Since $\cG|Y$ is minimal and $Y$ is infinite, every $\cG|Y$-orbit 
is infinite. 
We show that the germ on $Y$ induced by each arrow of $\cG|Y$ 
is realized by an element of $\St(\Gamma,Y)$. 
Let $g\in\cG|Y$ with $x:=s(g)\neq r(g)$. 
Choose $h\in\cG|Y$ with $s(h)=x$ and $r(h)\notin\{x,r(g)\}$. 
By regularity, there are compact open bisections $U,V$ of $\cG$ 
containing $g,h$, respectively, such that $YU=UY$ and $YV=VY$. 
After shrinking to a common source neighborhood of $x$, 
we may assume that $s(U)=s(V)$ and that 
$s(U)$, $r(U)$ and $r(V)$ are mutually disjoint. 
By Lemma~\ref{threecycle} (1)--(2), the three-cycle 
\[ 
\alpha:=\tau_V\tau_U=[\tau_U,\tau_V]\in\sD(\cG) 
\] 
agrees with $\theta_U$ on $s(U)$. 
Both transpositions preserve $Y$, so $\alpha\in\St(\Gamma,Y)$. 
Its restriction to $Y$ therefore has the required germ at $x$. 

If $s(g)=r(g)=x$, choose $h\in\cG|Y$ with 
$s(h)=x\neq r(h)$ and write $g=(gh^{-1})h$. 
Both factors have distinct source and range, so their germs 
are realized by the preceding construction. 
The product of the corresponding stabilizer elements 
then realizes the germ induced by $g$. 
Conversely, every element of $\St(\Gamma,Y)$ restricts 
to a homeomorphism of $Y$ induced by a full bisection of $\cG|Y$. 
Thus the two actions by local homeomorphisms have exactly 
the same germs, with the same topology. 
By Definition~\ref{def:effectivequotient}, 
their groupoid of germs is $\Eff(\cG|Y)=\cG{\downarrow}Y$. 
\end{proof}

For spatial realization, Lemma~\ref{involution} supplies 
the involutions needed for (F1) and (F2), 
while Lemma~\ref{threecycle} supplies the elements 
of order three needed for (F3). 

\begin{theorem}\label{spatial} 
For $i=1,2$, let $\cG_i$ be a minimal ample groupoid 
with $\cG_i^{(0)}$ a Cantor set, 
let $\Gamma_i$ be a subgroup of $\sF(\cG_i)$ containing $\sD(\cG_i)$, 
and let $Y_i\subset\cG_i^{(0)}$ be an infinite closed set 
that is nowhere dense in $\cG_i^{(0)}$. 
Suppose that $\St(\cG_i,Y_i)$ acts minimally on $Y_i$. 
Then every group isomorphism 
$\Phi:\St(\Gamma_1,Y_1)\to\St(\Gamma_2,Y_2)$ is spatially implemented, 
that is, there exists a homeomorphism 
$\phi:\cG_1^{(0)}\to\cG_2^{(0)}$ such that 
$\Phi(\alpha)=\phi\circ\alpha\circ\phi^{-1}$ holds 
for every $\alpha\in\St(\Gamma_1,Y_1)$. 
\end{theorem} 

\begin{proof} 
It suffices to verify that the group $\St(\Gamma_i,Y_i)$ is of class F 
in the sense of \cite[Definition 3.1]{Ma15crelle} 
and apply \cite[Theorem 3.5]{Ma15crelle} 
(see also Definition 6.3 and Theorem 6.6 of \cite{NO19IJM}). 
Let us omit the subscript $i$ and write $\Lambda:=\St(\Gamma,Y)$. 
All the elements produced below will lie in $\sD(\cG)$, 
so that they belong to $\Lambda$ as soon as they preserve $Y$. 

Condition (F0) is automatic, 
because $\Lambda$ is a subgroup of $\sF(\cG)$ 
(see \cite[Lemma 2.2]{Ma15crelle}). 

For (F1), take $x\in\cG^{(0)}$ and a clopen neighborhood $A$ of $x$. 
We first choose a bisection $U\in\cCO(\cG)$ and 
a clopen set $B_0\subset s(U)$ as follows. 

Assume first that $x$ is not in $Y$. 
Shrinking $A$, we may assume $A\cap Y=\emptyset$. 
Since $\cG$ is minimal, 
there exists a bisection $U\in\cCO(\cG)$ 
such that $x\in s(U)$ and $r(U)\sqcup s(U)\subset A$; 
we put $B_0:=s(U)$. 

Assume next that $x$ is in $Y$. 
By Proposition~\ref{stabilizerminimality}, $Y$ is a Cantor set. 
Therefore $(A\cap Y)\setminus\{x\}$ is 
a nonempty relatively open subset of $Y$, 
and the minimality of the action of $\St(\cG,Y)$ on $Y$ yields 
$\alpha=\theta_U\in\St(\cG,Y)$ with $\alpha(x)\in(A\cap Y)\setminus\{x\}$. 
Since $\alpha$ preserves $Y$, we have $YU=UY$. 
Choose a clopen neighborhood $B_0\subset s(U)$ of $x$ so that 
$B_0\sqcup\alpha(B_0)\subset A$. 

In either case, $x\in B_0\subset s(U)$, 
the sets $B_0$ and $\theta_U(B_0)$ are disjoint and 
their union is contained in $A$. 
By minimality there exists $g\in\cG$ 
with $s(g)=x$ and $r(g)\in B_0\setminus\{x\}$, 
and we may choose a bisection $Z\in\cCO(\cG)$ containing $g$ 
with $s(Z)\sqcup r(Z)\subset B_0$. 
Put $B_1:=s(Z)$, $B_2:=r(Z)$ and $V_j:=UB_j$. 
Then $r(V_j)\cap s(V_j)=\theta_U(B_j)\cap B_j=\emptyset$, 
the supports $B_1\sqcup\theta_U(B_1)$ and $B_2\sqcup\theta_U(B_2)$ 
are disjoint, and $Z$ has source $B_1$ and range $B_2$. 
By Lemma~\ref{involution}, 
$\beta:=\tau_{V_1}\tau_{V_2}$ is an involution in $\sD(\cG)$ with 
\[ 
x\in B_1\subset\supp(\beta) 
=B_1\sqcup\theta_U(B_1)\sqcup B_2\sqcup\theta_U(B_2)\subset A. 
\] 
In the first case $\supp(\beta)\subset A$ is disjoint from $Y$, 
and in the second case each $\tau_{V_j}$ preserves $Y$ 
because $YU=UY$; in both cases $\beta\in\Lambda$, and (F1) follows. 

To verify (F2), 
let $\alpha\in\Lambda\setminus\{1\}$ satisfy $\alpha^2=1$, 
let $W$ be the full bisection with $\alpha=\theta_W$, 
and let $A\subset\supp(\alpha)$ be a nonempty clopen set. 
Since $\supp(\alpha)$ is the closure of the set of points moved by 
$\alpha$, we can choose $x_0\in A$ with $\alpha(x_0)\neq x_0$. 
Choose a clopen neighborhood $B_0\subset A$ of $x_0$ such that 
$B_0\cap\alpha(B_0)=\emptyset$. 
As above, minimality provides a bisection $Z\in\cCO(\cG)$ with 
$s(Z)\sqcup r(Z)\subset B_0$; 
put $B_1:=s(Z)$, $B_2:=r(Z)$ and $V_j:=WB_j$. 
Since $\alpha^2=1$, the map $\tau_{V_j}$ agrees with $\alpha$ 
on $B_j\sqcup\alpha(B_j)$ and is the identity elsewhere. 
The four sets $B_1,\alpha(B_1),B_2,\alpha(B_2)$ are mutually disjoint, 
and $Z$ has source $B_1$ and range $B_2$, 
so Lemma~\ref{involution} shows that 
$\beta:=\tau_{V_1}\tau_{V_2}$ is an involution in $\sD(\cG)$. 
Its support is $B_1\sqcup\alpha(B_1)\sqcup B_2\sqcup\alpha(B_2)$, 
which is contained in $A\cup\alpha(A)$, 
and $\beta$ agrees with $\alpha$ on its support. 
Finally, $\beta$ preserves $Y$ because $\alpha$ does and 
$\alpha$ maps $\supp(\beta)$ onto itself. 
Hence $\beta\in\Lambda\setminus\{1\}$ verifies (F2). 

For (F3), let $A\subset\cG^{(0)}$ be a nonempty clopen set. 
Since $Y$ is nowhere dense, 
we may choose a nonempty clopen subset $B\subset A\setminus Y$. 
By minimality of $\cG$, we can choose bisections $U,V\in\cCO(\cG)$ 
with $s(U)=s(V)$ such that $s(U)$, $r(U)$ and $r(V)$ are nonempty, 
mutually disjoint subsets of $B$. 
By Lemma~\ref{threecycle} (1), 
the element $\tau_V\tau_U$ belongs to $\sD(\cG)$, 
has order three and has support contained in $B$. 
As $B\cap Y=\emptyset$, it belongs to $\Lambda$ and verifies (F3). 
\end{proof}

The next theorem shows that 
the abstract isomorphism class of a stabilizer records 
two kinds of information: 
the Kakutani equivalence class of the ambient groupoid and, 
when the effective reductions are Hausdorff, 
the isomorphism class of the effective reduction to the closed set. 
These provide two obstructions to isomorphism between stabilizers, 
arising respectively from the action away from the closed set 
and the action on it. 

\begin{theorem}\label{thm:stabilizerinvariants} 
For $i=1,2$, let $\cG_i$ be a minimal ample groupoid 
with $\cG_i^{(0)}$ a Cantor set, 
let $\Gamma_i$ be a subgroup of $\sF(\cG_i)$ containing $\sD(\cG_i)$, 
and let $Y_i\subset\cG_i^{(0)}$ be an infinite closed set 
that is nowhere dense in $\cG_i^{(0)}$. 
Suppose that $\St(\cG_i,Y_i)$ acts minimally on $Y_i$. 
If $\St(\Gamma_1,Y_1)$ is isomorphic to $\St(\Gamma_2,Y_2)$, 
the following hold. 
\begin{enumerate} 
\item $\cG_1$ is Kakutani equivalent to $\cG_2$. 
\item Suppose further that each $Y_i$ satisfies 
\eqref{eq:hausdorffreduction} with respect to $\cG_i$. 
Then $\cG_1{\downarrow}Y_1$ is isomorphic to $\cG_2{\downarrow}Y_2$. 
\end{enumerate} 
\end{theorem} 

\begin{proof} 
Let $\Phi:\St(\Gamma_1,Y_1)\to\St(\Gamma_2,Y_2)$ be an isomorphism. 
It follows from Theorem~\ref{spatial} that 
there exists a homeomorphism $\phi:\cG_1^{(0)}\to\cG_2^{(0)}$ 
such that $\Phi(\alpha)=\phi\circ\alpha\circ\phi^{-1}$ holds 
for every $\alpha\in\St(\Gamma_1,Y_1)$. 

For each $i=1,2$, the set $Y_i$ is the largest proper closed subset 
invariant under $\St(\Gamma_i,Y_i)$. 
Indeed, for every nonempty clopen $D\subset\cG_i^{(0)}\setminus Y_i$, 
the local subgroup $\sD(\cG_i)_D\subset\St(\Gamma_i,Y_i)$ acts 
minimally on $D$ by Lemma~\ref{threecycle} (3). 
Since $Y_i$ is nowhere dense, 
every orbit outside $Y_i$ is therefore dense in $\cG_i^{(0)}$. 
It follows that $\phi(Y_1)=Y_2$. 

(1)\: 
Pick a nonempty clopen subset $A\subset\cG_1^{(0)}$ 
so that $A\cap Y_1=\emptyset$. 
Since $\supp(\Phi(\alpha))=\phi(\supp(\alpha))$ 
for $\alpha\in\St(\Gamma_1,Y_1)$, we have 
\[ 
(\Gamma_1)_A=\St(\Gamma_1,Y_1)_A 
\cong\St(\Gamma_2,Y_2)_{\phi(A)}=(\Gamma_2)_{\phi(A)}, 
\] 
where the equalities follow from $A\cap Y_1=\emptyset$ and 
$\phi(A)\cap Y_2=\emptyset$. 
For $i=1,2$ the reduction $\cG_i|A_i$, where $A_1:=A$ and 
$A_2:=\phi(A)$, is a minimal ample groupoid 
whose unit space is a Cantor set, 
and $A_i$ is full by minimality of $\cG_i$. 
Under our identification of topological full groups with local subgroups, 
we have 
\[ 
\sD(\cG_i|A_i)\subset(\Gamma_i)_{A_i}\subset\sF(\cG_i|A_i). 
\] 
It therefore follows from \cite[Theorem 3.9]{Ma15crelle} that 
$\cG_1|A$ is isomorphic to $\cG_2|\phi(A)$, 
and so $\cG_1$ is Kakutani equivalent to $\cG_2$. 

(2)\: 
By Proposition~\ref{stabilizerminimality}, each regular reduction 
$\cG_i|Y_i$ is minimal. 
Lemma~\ref{stabilizergerms} therefore identifies 
$\cG_i{\downarrow}Y_i$ with the groupoid of germs 
of the action of $\St(\Gamma_i,Y_i)$ on $Y_i$. 
Since $\phi(Y_1)=Y_2$ and $\Phi$ is implemented by $\phi$, 
the homeomorphism $\phi|Y_1:Y_1\to Y_2$ conjugates these actions. 
It consequently induces an isomorphism of their groupoids of germs, 
and hence an isomorphism 
$\cG_1{\downarrow}Y_1\cong\cG_2{\downarrow}Y_2$. 
\end{proof} 

\begin{corollary}\label{cor:stabilizerinvariantsFD} 
Let $\cG_i$, $Y_i$ be as in Theorem~\ref{thm:stabilizerinvariants}. 
If $\St(\cG_1,Y_1)\cong\St(\cG_2,Y_2)$, or 
if $\St(\cG_1,Y_1)\cap\sD(\cG_1)\cong\St(\cG_2,Y_2)\cap\sD(\cG_2)$, 
then the conclusions of 
Theorem~\ref{thm:stabilizerinvariants} hold. 
\end{corollary} 

\begin{remark}\label{rem:spatialRubin} 
The usual form of Rubin's spatial realization theorem 
for locally dense actions \cite{Ru89TAMS} does not apply directly 
to the action in Theorem~\ref{spatial}. 
Indeed, put $X:=\cG^{(0)}$ and $\Lambda:=\St(\Gamma,Y)$. 
For $x\in Y$ and a clopen neighborhood $U\subset X$ of $x$, 
the local subgroup $\Lambda_U$ satisfies $\overline{\Lambda_Ux}\subset Y$, 
so this orbit closure has empty interior in $X$. 
Thus the local density hypothesis fails at every point of $Y$. 

The pair $(\Lambda,X)$ is nevertheless locally moving and 
locally flexible in the sense of \cite[Definition 6.8]{NO19IJM}. 
Local movability follows from condition (F1) above. 
As for local flexibility, since $Y$ is nowhere dense, 
every nonempty open subset of $X$ contains a nonempty clopen set $B$ 
with $B\cap Y=\emptyset$, and then $\sD(\cG)_B\subset\Lambda$; 
given nonempty open sets $C_1,C_2\subset B$, 
Lemma~\ref{threecycle} (3) provides 
$\sigma\in\sD(\cG)_B$ with $\sigma(C_1)\cap C_2\neq\emptyset$, 
so that $B$ is flexible. 
By \cite[Theorem 6.11]{NO19IJM}, 
every isomorphism $\St(\Gamma_1,Y_1)\to\St(\Gamma_2,Y_2)$ 
is therefore induced by an isomorphism 
between the Boolean algebras of regular open subsets of $\cG_1^{(0)}$ 
and $\cG_2^{(0)}$. 
This does not by itself produce a homeomorphism, 
and the stronger statements of \cite[Theorem 6.19]{NO19IJM} and 
\cite[Theorem 3.50(a)]{Ru89TAMS} require in addition that 
every clopen set be recognizable in the sense of 
\cite[Definition 6.13]{NO19IJM}, 
equivalently strongly recognizably clopen in the sense of 
\cite[Definition 3.44]{Ru89TAMS}. 
For $\Gamma=\sF(\cG)$ this holds, 
because both restriction to an invariant clopen set and 
exchange of disjoint clopen sets preserve membership in $\St(\cG,Y)$, 
and the Cantor space $X$ satisfies the remaining topological conditions 
defining Rubin's class $K^{\mathrm{TD}}$ 
(see \cite[Proposition 3.49(a)]{Ru89TAMS}); 
thus \cite[Theorem 6.19]{NO19IJM} gives an alternative proof of 
Theorem~\ref{spatial} in that case. 
For a general $\Gamma$ these cut-and-paste operations need not 
preserve $\Gamma$, which is why we use the class F formulation. 
\end{remark} 

\begin{remark}\label{rem:stabilizermeasures} 
In the setting of Theorem~\ref{thm:stabilizerinvariants}, with 
$\cG_1=\cG_2=\cG$, invariant measures also give obstructions 
to conjugacy inside $\sF(\cG)$. 
Indeed, if $\alpha\in\sF(\cG)$ conjugates $\St(\cG,Y_1)$ 
onto $\St(\cG,Y_2)$, the proof of the theorem gives 
$\alpha(Y_1)=Y_2$. 
Since $\alpha$ preserves every $\cG$-invariant probability measure 
$\mu$, we have $\mu(Y_1)=\mu(Y_2)$. 
Thus the function $\mu\mapsto\mu(Y)$ on the invariant probability 
measures is an invariant of such conjugacy. 
For abstract group isomorphisms, an additional argument would be 
needed to identify the ambient invariant measures under the spatial 
implementing homeomorphism. 
We do not pursue these measure-theoretic questions here. 
\end{remark}

\subsection{Realization by effective reductions} 
\label{subsec:effectiverealization} 

We ask which effective ample groupoids can occur 
as reductions to closed nowhere dense subsets of a given groupoid. 
Besides its independent interest, 
this realization problem supplies examples of maximal stabilizers 
through Theorem~\ref{stabilizermaximal}. 
We distinguish effective reductions 
from full restrictions to \'etale subsets. 
Our main construction proceeds in two stages: 
we first realize the UHF groupoid of type $2^\infty$ 
in an arbitrary minimal ample groupoid, 
and then realize any prescribed AF groupoid inside it. 
We begin with the compatibility of reductions 
and the classical realization result for AF groupoids. 

\begin{definition}\label{def:nowheredenserealization} 
Let $\cH$ be an effective ample groupoid with compact unit space. 
We say that $\cH$ is realizable in $\cG$ if there exists 
a closed nowhere dense subset 
$Y\subset\cG^{(0)}$ satisfying \eqref{eq:hausdorffreduction} 
such that 
\[ 
\cH\cong\cG{\downarrow}Y. 
\] 
We say that $\cH$ is strongly realizable in $\cG$ 
if there exists a closed nowhere dense $\cG$-\'etale subset 
$Y\subset\cG^{(0)}$ such that 
\[ 
\cH\cong\cG|Y=r^{-1}(Y)\cap s^{-1}(Y), 
\] 
where the full restriction has the relative topology. 
\end{definition} 

Strong realizability implies realizability, since the full restriction 
in the definition, being isomorphic to $\cH$, is effective and hence 
equals $\cG{\downarrow}Y$. 
The requirement that $Y$ be $\cG$-\'etale alone does not ensure 
effectiveness of the full restriction; 
see Example~\ref{ex:fullrestriction} (2). 

We first record how effective reductions behave under further reduction. 

\begin{proposition}\label{prop:iteratedreduction} 
Let $Y,Z\subset\cG^{(0)}$ be closed subsets with $Z\subset Y$, 
and assume that both satisfy \eqref{eq:hausdorffreduction}. 
Then $Z$ satisfies \eqref{eq:hausdorffreduction} with respect to 
$\cG{\downarrow}Y$, and there is a canonical identification 
\[ 
(\cG{\downarrow}Y){\downarrow}Z\subset\cG{\downarrow}Z 
\] 
with a wide open subgroupoid, inducing the identity on $Z$. 
If $Y$ is $\cG$-\'etale, this identification is an isomorphism. 
\end{proposition} 

\begin{proof} 
Put $\cK:=(\cG|Y)|Z$, using regular reductions at both steps. 
The sets $UZ$, where $U$ is a compact open bisection of $\cG$ 
satisfying $YU=UY$ and $ZU=UZ$, form a basis for $\cK$. 
Indeed, a bisection of $\cG|Y$ is locally the restriction of 
a compact open bisection $U$ of $\cG$ preserving $Y$. 
If its restriction locally preserves $Z$, we may shrink $U$ 
so that $ZU=UZ$ as well, since $Z\subset Y$. 
Each such $UZ$ is also an open bisection of $\cG|Z$. 
Thus $\cK$ is a wide open subgroupoid of $\cG|Z$. 

The quotient map $\cG|Y\to\cG{\downarrow}Y$ is the identity 
on $Y$ and maps open bisections homeomorphically onto open 
bisections. 
Consequently $\cK$ and $(\cG{\downarrow}Y)|Z$ induce the same 
germs of local homeomorphisms on $Z$, giving a canonical isomorphism 
$\Eff(\cK)\cong\Eff\left((\cG{\downarrow}Y)|Z\right)$. 
Since $\cK$ is open in $\cG|Z$, we have 
$\Iso(\cK)^\circ=\cK\cap\Iso(\cG|Z)^\circ$. 
It follows that the inclusion $\cK\subset\cG|Z$ descends to 
an open injective homomorphism 
$\Eff(\cK)\to\Eff(\cG|Z)=\cG{\downarrow}Z$, 
which is the identity on $Z$. 
This proves both the Hausdorffness assertion and the claimed 
identification with a wide open subgroupoid. 

Suppose now that $Y$ is $\cG$-\'etale. 
By \eqref{eq:etalesubsetregularity}, every arrow of $\cG$ 
with source and range in $Y$ is regular for $Y$. 
Given $g\in\cG|Z$, choose bisections $U,V\in\cCO(\cG)$ 
containing $g$ such that $YU=UY$ and $ZV=VZ$. 
Their intersection preserves both $Y$ and $Z$, so $g\in\cK$. 
Thus $\cK=\cG|Z$, and the induced map of effective quotients 
is an isomorphism. 
\end{proof} 

The full-restriction version of the \'etale case of 
Proposition~\ref{prop:iteratedreduction} gives transitivity of 
strong realizability. 
Indeed, if $Y$ is a closed $\cG$-\'etale subset and 
$Z\subset Y$ is a closed $(\cG|Y)$-\'etale subset, then 
\[ 
(\cG|Y)|Z=r^{-1}(Z)\cap s^{-1}(Z)=\cG|Z 
\] 
with the relative topology, so $Z$ is $\cG$-\'etale. 
If $Y$ is nowhere dense in $\cG^{(0)}$, so is $Z$. 
Applying this observation to two strong realizations proves the claim. 

\begin{proposition}\label{prop:AFregularreduction} 
Let $\cG$ be an AF groupoid with compact unit space $X$, 
and let $Y\subset X$ be closed. 
Then $\cG{\downarrow}Y=\cG|Y$ is an AF groupoid. 
\end{proposition} 

\begin{proof} 
Since $\cG$ is principal, so is $\cG|Y$, and hence 
$\cG{\downarrow}Y=\cG|Y$. 
Choose an increasing sequence of elementary subgroupoids 
$\cG_n$ with union $\cG$. 
The ample groupoid $\cG|Y$ admits an increasing compact open 
exhaustion $(K_m)$ with $Y\subset K_m=K_m^{-1}$. 
For each $m$, compactness gives $n$ such that $K_m\subset\cG_n$. 
Since $\cG_n$ is compact and \'etale, its orbits have cardinality 
bounded by some integer $N$. 
Every arrow in the subgroupoid generated by $K_m$ is a product 
along a path in one of these orbits. 
Principality allows repeated vertices to be removed from this path, 
so at most $N-1$ factors suffice. 
Thus the generated subgroupoid is a finite union of compact open 
sets $K_m^j$ (including $K_m^0=Y$), and is elementary in $\cG|Y$. 
These subgroupoids increase with $m$ and exhaust $\cG|Y$, 
which is therefore AF. 
\end{proof} 

We next record a classical strong realization result 
from the Bratteli diagram description of AF groupoids. 
Following \cite[Definition 4.11]{GPS04ETDS}, we say that a closed 
subset $Y\subset\cG^{(0)}$ is $\cG$-thin if $\mu(Y)=0$ 
for every $\cG$-invariant probability measure $\mu$. 

\begin{proposition}\label{prop:AFrealization} 
Let $\cG$ be a minimal AF groupoid with Cantor unit space, 
and let $\cH$ be an AF groupoid with compact unit space. 
Then $\cH$ is strongly realizable in $\cG$: 
there exists a closed nowhere dense $\cG$-\'etale subset 
$Y\subset\cG^{(0)}$ such that 
\[ 
\cH\cong\cG{\downarrow}Y=\cG|Y=r^{-1}(Y)\cap s^{-1}(Y). 
\] 
Moreover, $Y$ can be chosen $\cG$-thin. 
\end{proposition} 

\begin{proof} 
Apply \cite[Lemma 3.1]{Ma08JMSJ}, 
taking the closed set whose saturation is to be avoided to be empty. 
It gives a closed $\cG$-\'etale subset $Y\subset\cG^{(0)}$ 
and an isomorphism from $\cH$ onto the full restriction of $\cG$ 
to $Y$, with its relative topology, such that $\mu(Y)=0$ 
for every $\cG$-invariant probability measure $\mu$. 
Such measures exist and have full support 
because $\cG$ is minimal and AF, so $Y$ is nowhere dense. 
By \eqref{eq:etalesubsetregularity}, the full restriction is $\cG|Y$. 
It is principal, and hence also equals $\cG{\downarrow}Y$. 
\end{proof} 

Together with Proposition~\ref{prop:AFregularreduction}, 
this shows that, for a minimal AF groupoid $\cG$ with Cantor unit space 
and an effective ample groupoid $\cH$ with compact unit space, 
the following conditions are equivalent: $\cH$ is AF; 
$\cH$ is realizable in $\cG$; 
$\cH$ is strongly realizable in $\cG$. 

For a fixed minimal AF groupoid, thin strong realizations 
of a given AF groupoid are unique up to ambient automorphisms. 
When the restrictions are minimal, their dimension groups, 
with the order units, also classify the resulting stabilizers. 

\begin{corollary}\label{cor:thinAFrealization} 
Let $\cG$ be a minimal AF groupoid with Cantor unit space, 
and let $Y_1,Y_2\subset\cG^{(0)}$ be closed, $\cG$-thin, 
$\cG$-\'etale Cantor sets. 
\begin{enumerate} 
\item Every isomorphism $\cG|Y_1\to\cG|Y_2$ extends to 
an automorphism of $\cG$. 
\item Suppose that $\cG|Y_1$ and $\cG|Y_2$ are minimal, 
and let $\Gamma$ be either $\sF(\cG)$ or $\sD(\cG)$. 
Then $\St(\Gamma,Y_1)$ and $\St(\Gamma,Y_2)$ are isomorphic 
if and only if there is an order isomorphism 
\[ 
H_0(\cG|Y_1)\longrightarrow H_0(\cG|Y_2) 
\] 
sending the order unit $[1_{Y_1}]$ to $[1_{Y_2}]$. 
\end{enumerate} 
\end{corollary} 

\begin{proof} 
Part (1) is \cite[Lemma 4.15]{GPS04ETDS}. 
For (2), the sets $Y_i$ are nowhere dense by thinness, since 
$\cG$-invariant probability measures exist and have full support. 
Proposition~\ref{stabilizerminimality} gives minimality of the 
stabilizer actions on $Y_i$. 
If the stabilizers are isomorphic, 
Theorem~\ref{thm:stabilizerinvariants} (2) gives an isomorphism 
$\cG|Y_1\cong\cG|Y_2$, which induces the required order isomorphism. 
Conversely, the groupoids $\cG|Y_i$ are AF by 
Proposition~\ref{prop:AFregularreduction}, and their ordered $H_0$ 
groups with order units are their dimension groups 
\cite[Theorem 4.10]{Ma12PLMS}. 
The given order isomorphism therefore yields an isomorphism 
$\cG|Y_1\cong\cG|Y_2$ by \cite[Lemma 4.13]{GPS04ETDS}. 
By (1), it extends to an automorphism of $\cG$. 
The induced homeomorphism of $\cG^{(0)}$ normalizes both 
$\sF(\cG)$ and $\sD(\cG)$ and sends $Y_1$ onto $Y_2$, 
so it conjugates the two stabilizers. 
\end{proof} 

We now extend this strong realization result to 
arbitrary minimal ample groupoids. 
Let $\mathcal R_{2^\infty}$ denote the UHF groupoid of type $2^\infty$, 
namely the tail equivalence relation on $\{0,1\}^{\N}$ 
with its usual AF topology. 
The construction below uses a binary family of bisections to produce 
the finite-stage tail equivalence relations. 
By shrinking their common source along a compact exhaustion of 
$\cG$, we exclude all additional arrows from the full restriction 
to the limiting Cantor set. 

\begin{proposition}\label{prop:UHFrealization} 
Let $\cG$ be a minimal effective ample groupoid 
with Cantor unit space $X$. 
Then $\mathcal R_{2^\infty}$ is strongly realizable in $\cG$: 
there exists a nowhere dense $\cG$-\'etale Cantor set $Y\subset X$ 
such that $\cG|Y=r^{-1}(Y)\cap s^{-1}(Y)\cong\mathcal R_{2^\infty}$, 
where the full restriction carries the relative topology from $\cG$. 
\end{proposition} 

\begin{proof} 
Since $\cG$ is second countable and effective, 
choose $x\in X$ with $\cG_x=\{x\}$. 
We first record a finite-stage observation. 
Suppose that $V_1,\dots,V_m\in\cCO(\cG)$ are bisections 
with common source $U\ni x$ and pairwise disjoint ranges. 
Put 
\[ 
\cK:=\bigsqcup_{i,j=1}^m V_iV_j^{-1},\qquad 
F:=\{\theta_{V_i}(x)\mid 1\leq i\leq m\}. 
\] 
For any compact subset $C\subset\cG$, 
the compact set $C\setminus\cK$ contains no arrow 
with both source and range in $F$: 
there is a unique arrow between any two points of the orbit of $x$. 
By compactness, there is a clopen neighborhood $A\subset U$ of $x$ 
such that, writing $Z_A:=\bigsqcup_i r(V_iA)$, 
\[ 
C\cap r^{-1}(Z_A)\cap s^{-1}(Z_A) 
\subset\bigsqcup_{i,j=1}^m V_iAV_j^{-1}. 
\] 

Fix a compatible metric on $X$ and compact subsets 
$C_1\subset C_2\subset\cdots$ covering $\cG$. 
We construct neighborhoods $U_n\in\cCO(X)$ of $x$ and 
bisections $V_\mu\in\cCO(\cG)$, indexed by $\mu\in\{0,1\}^n$, 
with common source $U_n$ and pairwise disjoint ranges. 
Write 
\[ 
Z_n:=\bigsqcup_{\lvert\mu\rvert=n}r(V_\mu),\qquad 
\cK_n:=\bigsqcup_{\lvert\mu\rvert=\lvert\nu\rvert=n}V_\mu V_\nu^{-1}. 
\] 
Start with $U_0=X$ and $V_{\emptyset}=X$. 
Given the construction at level $n$, minimality provides an arrow 
$g_n$ with $s(g_n)=x$ and $r(g_n)\in U_n\setminus\{x\}$. 
Choose a compact open bisection $W_n$ containing $g_n$ 
whose source and range are disjoint subsets of $U_n$. 
For a sufficiently small clopen neighborhood 
$U_{n+1}\subset s(W_n)$ of $x$, put 
\[ 
V_{\mu0}:=V_\mu U_{n+1},\qquad 
V_{\mu1}:=V_\mu W_nU_{n+1} 
\qquad (\lvert\mu\rvert=n). 
\] 
Applying the observation to these bisections and $C_{n+1}$, 
and shrinking $U_{n+1}$ further, we arrange that 
\begin{equation}\label{eq:UHFcontrol} 
C_{n+1}\cap r^{-1}(Z_{n+1})\cap s^{-1}(Z_{n+1}) 
\subset\cK_{n+1}. 
\end{equation} 
We also require that every $r(V_{\mu i})$ has diameter less than 
$2^{-n-1}$ and that 
\[ 
U_{n+1}\cup r(W_nU_{n+1})\subsetneq U_n. 
\] 
These requirements can all be met by shrinking the common source. 

Set $Y:=\bigcap_{n\geq0}Z_n$. 
The nested binary partitions and the diameter condition give 
a homeomorphism $h:\{0,1\}^{\N}\to Y$, 
with $h(\mu\xi)\in r(V_\mu)$ for every finite word $\mu$ 
and infinite word $\xi$. 
Each parent range contains a nonempty clopen subset omitted 
by its two children. 
Together with the diameter condition, this shows that $Y$ is 
nowhere dense in $X$. 
The common choice of $W_n$ for all words of length $n$ gives 
\[ 
\theta_{V_\mu V_\nu^{-1}}(h(\nu\xi))=h(\mu\xi) 
\qquad (\lvert\mu\rvert=\lvert\nu\rvert). 
\] 
Consequently, under $h$, the compact groupoid 
\[ 
R_n:=\cK_n\cap r^{-1}(Y)\cap s^{-1}(Y) 
\] 
is the usual finite-stage relation on $\{0,1\}^{\N}$ 
in which two sequences agree after the first $n$ coordinates. 
The construction also gives the equality of subsets of $\cG$ 
\[ 
(V_\mu V_\nu^{-1})Y 
=\bigsqcup_{i=0,1}(V_{\mu i}V_{\nu i}^{-1})Y 
\qquad (\lvert\mu\rvert=\lvert\nu\rvert=n), 
\] 
because the source cylinder in $Y$ is partitioned into its two children. 
Hence $R_n\subset R_{n+1}$ as subgroupoids of $\cG$. 
Equation~\eqref{eq:UHFcontrol} and $\bigcup_n C_n=\cG$ imply 
\[ 
r^{-1}(Y)\cap s^{-1}(Y)=\bigcup_{n\geq0}R_n. 
\] 
Each $R_n$ is compact open in the full restriction, 
and its relative topology agrees with that of the corresponding 
finite-stage relation. 
Thus the full restriction, with its relative topology, 
is isomorphic to $\mathcal R_{2^\infty}$. 
It is therefore \'etale, and \eqref{eq:etalesubsetregularity} 
gives $\cG|Y=r^{-1}(Y)\cap s^{-1}(Y)$. 
\end{proof} 

\begin{corollary}\label{cor:AFrealizationminimal} 
Let $\cG$ be a minimal effective ample groupoid 
with Cantor unit space $X$, and let $\cH$ be an AF groupoid 
with compact unit space. 
Then $\cH$ is strongly realizable in $\cG$: 
there exists a closed nowhere dense $\cG$-\'etale subset 
$Y\subset X$ such that 
\[ 
\cH\cong\cG|Y=r^{-1}(Y)\cap s^{-1}(Y), 
\] 
with the relative topology on the full restriction. 
\end{corollary} 

\begin{proof} 
By Proposition~\ref{prop:UHFrealization}, choose a nowhere dense 
$\cG$-\'etale Cantor set $Y_0\subset X$ 
such that $\cG|Y_0\cong\mathcal R_{2^\infty}$. 
Apply Proposition~\ref{prop:AFrealization} to the minimal AF 
groupoid $\cG|Y_0$. 
It gives a closed $(\cG|Y_0)$-\'etale subset $Y\subset Y_0$ 
such that $\cH\cong(\cG|Y_0)|Y$. 
Both reductions are full restrictions with the relative topology. 
Hence $Y$ is $\cG$-\'etale and $\cH\cong(\cG|Y_0)|Y=\cG|Y$. 
Since $Y\subset Y_0$, it is nowhere dense in $X$. 
\end{proof} 

\begin{remark}\label{rem:AFrealizationGV} 
The nested clopen construction in 
Proposition~\ref{prop:UHFrealization} is closely related to 
\cite[Lemma 7.1 and Proposition 7.2]{GV26ETDS}. 
For a minimal fully amplified group acting on a Cantor set, 
Grigorchuk and Vorobets construct uncountably many nowhere dense 
Cantor subsets on which the corresponding stabilizers act minimally. 
In our construction, \eqref{eq:UHFcontrol} also controls every 
ambient arrow with endpoints in $Y$. 
This identifies the full restriction with $\mathcal R_{2^\infty}$ 
and makes $Y$ \'etale. 
Corollary~\ref{cor:AFrealizationminimal} then allows any AF groupoid 
with compact unit space to be prescribed as the full restriction. 
Proposition~\ref{prop:uncountablestabilizers} uses these restrictions 
to distinguish the stabilizers up to abstract group isomorphism 
for any minimal ambient groupoid. 
If the ambient groupoid is almost finite or purely infinite, 
the stabilizers are also maximal. 
In comparison, 
\cite[Proposition 7.2]{GV26ETDS} constructs pairwise distinct closed 
subsets, without asserting that their stabilizers are pairwise 
nonisomorphic. 
\end{remark} 

Let $T$ be a homeomorphism of a Cantor set $X$. 
Its transformation groupoid consists of the arrows 
$(T^n(x),n,x)$, with $x\in X$ and $n\in\Z$, 
range $T^n(x)$, source $x$, and multiplication given by addition 
of the integer coordinates. 
Its topology is the product topology under the identification 
with $X\times\Z$, where $\Z$ is discrete. 
For minimal $\Z$ systems, the AF realization assertion of 
Corollary~\ref{cor:AFrealizationminimal} also follows directly 
from the Bratteli--Vershik model, as in 
Proposition~\ref{prop:AFrealization}. 
In this case, only AF groupoids can be realized. 

\begin{proposition}\label{prop:ZAFrealization} 
Let $\cG$ be the transformation groupoid of a minimal homeomorphism 
$T$ of a Cantor set $X$. 
For every closed nowhere dense subset $Y\subset X$, 
the groupoid $\cG{\downarrow}Y=\cG|Y$ is AF. 
Consequently, an effective ample groupoid $\cH$ with compact unit space 
is realizable in $\cG$ if and only if it is AF; 
in that case, it is strongly realizable in $\cG$. 
\end{proposition} 

\begin{proof} 
Let $Y\subset X$ be closed and nowhere dense. 
By the Baire category theorem, there exists 
\[ 
x\in X\setminus\bigcup_{n\in\Z}T^n(Y). 
\] 
The Bratteli--Vershik model gives a minimal AF wide open subgroupoid 
$\cK\subset\cG$ which splits the $T$-orbit of $x$ into the two 
$\cK$-orbits of $x$ and $T(x)$ and leaves all other orbits unchanged 
\cite[Theorem 1.2]{GPS04ETDS}. 
Since the $T$-orbit of $x$ misses $Y$, the full restrictions 
of $\cK$ and $\cG$ to $Y$ coincide. 
Openness of $\cK$ then gives $\cG|Y=\cK|Y$, 
which is AF by Proposition~\ref{prop:AFregularreduction}. 
Principality gives $\cG{\downarrow}Y=\cG|Y$. 

Conversely, let $\cH$ be an AF groupoid with compact unit space. 
Choose any $x\in X$ and form $\cK$ as above. 
The finite set $E:=\{x,T(x)\}$ is $\cK$-\'etale and 
has measure zero for every $\cK$-invariant probability measure, 
since these measures are atomless. 
Its $\cK$-saturation is precisely the $T$-orbit of $x$. 
Applying \cite[Lemma 3.1]{Ma08JMSJ} with $E$ as the closed set 
whose saturation is to be avoided gives a closed $\cK$-\'etale 
subset $Z\subset X$ with $\cK|Z\cong\cH$ which misses that orbit. 
Moreover, $Z$ has measure zero for every $\cK$-invariant probability 
measure, so it is nowhere dense by minimality of $\cK$. 
The full restrictions of $\cK$ and $\cG$ to $Z$ coincide 
with their relative topologies. 
Thus $Z$ is $\cG$-\'etale and 
\[ 
\cH\cong\cG{\downarrow}Z=\cG|Z=r^{-1}(Z)\cap s^{-1}(Z), 
\] 
which proves strong realizability independently of 
Corollary~\ref{cor:AFrealizationminimal}. 
\end{proof} 

For SFT groupoids, \cite[Proposition 8.7]{Ma26maximal3} shows 
that every SFT groupoid with Cantor unit space is strongly 
realizable in every irreducible SFT groupoid with Cantor unit space. 
Determining realizability and strong realizability for other classes 
of ample groupoids remains an interesting problem. 

Varying the minimal AF groupoid in 
Corollary~\ref{cor:AFrealizationminimal} gives 
the following family of pairwise nonisomorphic stabilizers. 
Under the almost finite or purely infinite hypothesis, 
the maximality criterion applies to every member of the family. 

\begin{proposition}\label{prop:uncountablestabilizers} 
Let $\cG$ be a minimal ample groupoid with Cantor unit space $X$. 
There exists a family $(Y_t)_{t\in\mathcal T}$, 
with $\#\mathcal T=2^{\aleph_0}$, 
of nowhere dense $\cG$-\'etale Cantor subsets of $X$ 
such that each $\cG|Y_t$ is a UHF groupoid and, 
for every subgroup $\sD(\cG)\subset\Gamma\subset\sF(\cG)$, 
the groups $\St(\Gamma,Y_t)$, $t\in\mathcal T$, are pairwise 
nonisomorphic. 
If $\cG$ is either almost finite or purely infinite, 
these groups are maximal subgroups of $\Gamma$. 
\end{proposition} 

\begin{proof} 
For each nonempty set $P$ of primes, let $\cH_P$ 
be the UHF groupoid of supernatural type $\prod_{p\in P}p^\infty$. 
Its one-vertex Bratteli diagram gives 
\[ 
H_0(\cH_P)\cong\Z[p^{-1}\mid p\in P]. 
\] 
Multiplication by a prime $p$ on this group is surjective 
precisely when $p\in P$. 
Thus the groupoids $\cH_P$ are pairwise nonisomorphic. 
By Corollary~\ref{cor:AFrealizationminimal}, choose a nowhere dense 
$\cG$-\'etale Cantor set $Y_P\subset X$ 
with $\cG|Y_P\cong\cH_P$. 
These restrictions are minimal and principal. 
Proposition~\ref{stabilizerminimality} therefore gives minimality 
of the stabilizer actions, and the effective reductions 
coincide with the prescribed AF groupoids. 
By Theorem~\ref{thm:stabilizerinvariants} (2), 
an isomorphism between $\St(\Gamma,Y_P)$ and 
$\St(\Gamma,Y_{P'})$ would imply 
$\cH_P\cong\cH_{P'}$, and hence $P=P'$. 
Taking $\mathcal T$ to be the set of nonempty sets of primes, 
which has cardinality $2^{\aleph_0}$, proves the nonisomorphism assertion. 
If $\cG$ is almost finite or purely infinite, 
Theorem~\ref{stabilizermaximal} also makes every $\St(\Gamma,Y_P)$ 
maximal in $\Gamma$. 
\end{proof}

The realization theorem also has a direct $C^*$-algebraic consequence. 
Recall that a $C^*$-subquotient of a $C^*$-algebra is a quotient of 
one of its $C^*$-subalgebras. 

\begin{corollary}\label{cor:AFsubquotients} 
Let $\cG$ be a minimal effective ample groupoid with Cantor unit space. 
For every AF groupoid $\cH$ with compact unit space, 
$C_r^*(\cH)$ is a $C^*$-subquotient of $C_r^*(\cG)$. 
More precisely, there exist a unital $C^*$-subalgebra 
$C(\cG^{(0)})\subset A\subset C_r^*(\cG)$ and a surjective 
homomorphism $\rho:A\to C_r^*(\cH)$ such that 
$\rho(C(\cG^{(0)}))=C(\cH^{(0)})$. 
In particular, the CAR algebra 
\[ 
M_{2^\infty}=\bigotimes_{n=1}^\infty M_2 
\] 
is a $C^*$-subquotient of $C_r^*(\cG)$. 
\end{corollary} 

\begin{proof} 
By Corollary~\ref{cor:AFrealizationminimal}, choose a closed 
nowhere dense $\cG$-\'etale subset $Y\subset\cG^{(0)}$ such that 
$\cG|Y\cong\cH$. 
The $Y$-preserving subgroupoid of 
Definition~\ref{def:preservingsubgroupoid} satisfies 
$\cS_Y=\cG|Y\sqcup\cG|(\cG^{(0)}\setminus Y)$ as sets. 
Since $\cS_Y$ is wide and open in $\cG$, extension by zero gives 
\[ 
C(\cG^{(0)})\subset A:=C_r^*(\cS_Y)\subset C_r^*(\cG). 
\] 
Since $Y$ is $\cS_Y$-invariant, restriction to $Y$ extends to a 
surjective homomorphism 
\[ 
C_r^*(\cS_Y)\longrightarrow C_r^*(\cS_Y|Y) 
=C_r^*(\cG|Y). 
\] 
Composing with the isomorphism induced by $\cG|Y\cong\cH$ gives 
a surjective homomorphism $\rho:A\to C_r^*(\cH)$. 
On the diagonal, this is the restriction map 
$C(\cG^{(0)})\to C(Y)$ followed by the identification 
$C(Y)\cong C(\cH^{(0)})$. 
In particular, $\rho(C(\cG^{(0)}))=C(\cH^{(0)})$. 
Thus $C_r^*(\cH)$ is a $C^*$-subquotient of $C_r^*(\cG)$ 
with the asserted diagonal compatibility. 
Taking $\cH=\mathcal R_{2^\infty}$ gives the last assertion, because 
$C_r^*(\mathcal R_{2^\infty})\cong M_{2^\infty}$. 
\end{proof} 

Glimm's theorem \cite{Gl61Annals} characterizes separable non-type I 
$C^*$-algebras by the occurrence of the CAR algebra as a 
$C^*$-subquotient. 
In the setting of Corollary~\ref{cor:AFsubquotients}, 
$C_r^*(\cG)$ is simple by Renault's simplicity theorem 
\cite[Corollary 4.6]{Re91JOT}. 
It is also separable, unital and infinite dimensional, and hence 
not of type I. 
Thus the existence of a CAR subquotient already follows abstractly 
from Glimm's theorem. 
The construction above specifies an intermediate algebra 
$C_r^*(\cS_Y)$ and a quotient induced by restriction to $Y$, 
for every prescribed AF groupoid $\cH$ with compact unit space. 
It also identifies the induced map on the canonical diagonal 
subalgebras. 
No amenability assumption on $\cG$ is needed: 
restriction to a closed invariant subset gives a surjection 
at the level of reduced groupoid $C^*$-algebras.

\section*{Acknowledgements}
The author used ChatGPT (GPT-6 Astra) and Claude (Opus 5.5) 
to assist with mathematical discussions, checking arguments, 
literature searches, and the drafting and revision of the manuscript. 
The mathematical arguments and references were independently 
verified by the author, who is responsible for the final content.

\bibliography{../zzz} 
\bibliographystyle{amsplain}

\end{document}